\documentclass[12pt,reqno]{article}

\usepackage[utf8]{inputenc}
\usepackage[T1]{fontenc}
\usepackage[english]{babel}
\usepackage{amssymb}
\usepackage{amsmath}
\usepackage{amsthm}
\usepackage{mathtools}
\usepackage{booktabs}
\usepackage{float}
\usepackage[hidelinks]{hyperref}
\usepackage{fullpage}

\newcommand{\doi}[1]{\href{https://doi.org/#1}{\nolinkurl{#1}}}

\newcommand{\seqnum}[1]{\href{https://oeis.org/#1}{\rm \underline{#1}}}
\newcommand{\N}{\mathbb{N}}
\newcommand{\tm}{\mathbf{t}}
\newcommand{\Z}{\mathbb{Z}}

\theoremstyle{plain}
\newtheorem{theorem}{Theorem}[section]
\newtheorem{lemma}[theorem]{Lemma}
\newtheorem{proposition}[theorem]{Proposition}
\newtheorem{corollary}[theorem]{Corollary}

\theoremstyle{definition}
\newtheorem{definition}[theorem]{Definition}
\newtheorem{example}[theorem]{Example}
\newtheorem{remark}[theorem]{Remark}

\begin{document}

\begin{center}
\vskip 1cm{\LARGE\bf The Thue-Morse transform}
\vskip .8cm
\large
Benoît Cloitre\\
\medskip
{\small\href{https://orcid.org/0009-0001-6778-153X}{ORCID: 0009-0001-6778-153X}}
\end{center}

\vskip .3in

\begin{abstract}
We define the Thue-Morse transform $\mathcal{T}$ on a class of infinite binary
words. It sends the alternating word $a_0 = 010101\dots$ to the
Thue-Morse sequence. We then study its orbit $a_m = \mathcal{T}^m(a_0)$
as well as the sequences $u_m$ and $v_m$ giving respectively the
positions of the ones and the zeros in $a_m$. We obtain an explicit
formula for $a_m$ and deduce Prouhet-Tarry-Escott identities for the
partitions induced by $u_m$ and $v_m$. We also give composition
formulas for $u_m$ and $v_m$, and a full description of the factor
complexity of $a_m$.
\end{abstract}
\bigskip

\noindent\emph{Keywords.} Thue-Morse sequence, Prouhet-Tarry-Escott
problem, factor complexity, composition laws, evil and odious numbers,
uniform morphism, automatic sequence.

\medskip
\noindent\emph{MSC 2020 Classification.} 11B85, 68R15, 11A63, 11D72, 11B83.

\bigskip

\section*{Notations and definitions}

A binary word is an infinite string over the alphabet $\{0, 1\}$,
indexed from zero. We write $w(n)$ for the letter at position $n$ of
a word $w$. The terms \emph{word} and \emph{sequence} are used
interchangeably, with \emph{word} preferred in combinatorial contexts
and \emph{sequence} in arithmetic contexts or when referring to an
OEIS entry.
For the Thue-Morse word, the positions of the $1$s and $0$s are the
classical odious and evil numbers. By analogy, for the words $a_m$
studied here, we call the positions of the $1$s and $0$s the
associated odious and evil numbers, denoted respectively by $u_m$
and $v_m$.

The set of natural numbers $\{0,1,2,\dots\}$ is denoted by $\N$.
The standard binomial coefficient is denoted by $\binom{n}{k}$.

The binary expansion of an integer $n \ge 0$ is denoted by
\[
  n = \sum_{p \ge 0} b_p(n)\, 2^p, \qquad b_p(n) \in \{0,1\}.
\]
On bits, the symbol $\oplus$ denotes the exclusive or (XOR) and
the symbol $\&$ denotes the logical and (AND).

We recall the definition of the Prouhet--Tarry--Escott problem. It asks
for two distinct finite sets \(A,B\subset \N\), with
\[
  |A|=|B|=n,
\]
such that
\[
  \sum_{x\in A} x^k=\sum_{y\in B} y^k
  \qquad \text{for } k=0,1,\ldots,d.
\]
For background, see Borwein and Ingalls~\cite{BorweinIngalls1994}.
The integer \(d\) is called the degree of the solution. The solution is
said to be ideal if \(n=d+1\). The Prouhet--Tarry--Escott pairs obtained
in Section~\ref{sec:PTE} are not ideal. For recent results on ideal
solutions, see~\cite{CMSV2024}.

\section{Introduction}

The Thue-Morse word $\tm(n)$ (\seqnum{A010060}), implicitly appearing
in Prouhet \cite{Prouhet1851} and explicitly studied by Thue
\cite{Thue1906,Thue1912}, begins with
\[
  0, 1, 1, 0, 1, 0, 0, 1, 1, 0, 0, 1, 0, 1, 1, 0,
  1, 0, 0, 1, 0, 1, 1, 0, 0, 1, 1, 0, 1, 0, 0, 1, \dots.
\]
Among its many properties, particularly those described in
\cite{AlloucheShallit1999}, the one of greatest interest for us here is the recurrence
\begin{align*}
\tm(2n) &= \tm(n) \\
\tm(2n+1) &= 1-\tm(n) \\
\tm(0) &= 0.
\end{align*}

Since the even and odd numbers partition $\N$, one may ask what the
sequence becomes if the partition is replaced by another while keeping
the same recurrence. The evil and
odious sequences $v$ and $u$, giving respectively the positions of
the $0$s and the $1$s in $\tm$, provide a first idea, defining a new
sequence $\tm'$ by
\begin{align*}
\tm'(v(n)) &= \tm'(n) \\
\tm'(u(n)) &= 1-\tm'(n) \\
\tm'(0) &= 0.
\end{align*}

The sequence obtained begins with
\[
0, 1, 0, 1, 1, 0, 1, 0, 0, 1, 0, 1, 1, 0, 1, 0,
1, 0, 1, 0, 0, 1, 0, 1, 1, 0, 1, 0, 0, 1, 0, 1, \dots,
\]
which is the OEIS entry \seqnum{A341389} \cite{OEIS}. The positions of its
$1$s and $0$s are respectively \seqnum{A158705} and \seqnum{A158704},
due to J.~Layman. In the comments to these entries,
Layman asks if they form a Prouhet-Tarry-Escott identity. We prove
that this is indeed the case in Corollary~\ref{cor:PTE-Layman}.
This question was the starting point of the present article. The
construction can in fact be iterated, yielding an infinite family of
new PTE partitions.

To formalize this iteration process, we introduce the Thue-Morse transform.

Let $w$ be a binary word starting with $01$ and containing infinitely many $0$s and infinitely many $1$s. Let $v_w(n)$ and $u_w(n)$ denote the positions of the $n$-th $0$ and the $n$-th $1$ of $w$, with $n$ running over $\N$. Thus $v_w(0)$ is the position of the first $0$ and $u_w(0)$ is the position of the first $1$. Since $w$ starts with $01$, we have $v_w(0) = 0$ and $u_w(0) = 1$.

We define a new word $w'$ as follows. First set
\[
  w'(0) = 0.
\]
Then impose, for every $n \ge 0$, the relations
\begin{align*}
  w'(v_w(n)) &= w'(n) \\
  w'(u_w(n)) &= 1 - w'(n).
\end{align*}

These relations determine $w'$ by induction. The first relation at $n = 0$ reads $w'(0) = w'(0)$ and is already satisfied. The second gives $w'(1) = 1 - w'(0) = 1$. For every $n \ge 1$, the interval $\{0, 1, \ldots, n\}$ contains at least one $0$ and at least one $1$, hence
\[
  v_w(n) > n, \qquad u_w(n) > n.
\]
The values of $w'$ at the positions $v_w(n)$ and $u_w(n)$ thus depend only on earlier values. The word $w'$ is therefore well defined.

The Thue-Morse transform is the map
\[
  \mathcal{T}(w) = w'.
\]

For every $n \ge 0$, the defining relations give
\[
  w'(v_w(n)) + w'(u_w(n)) = w'(n) + (1 - w'(n)) = 1.
\]
Each pair of positions $\{v_w(n), u_w(n)\}$ thus contains one $0$ and one $1$ in $w'$. Hence $w'$ contains infinitely many $0$'s and infinitely many $1$'s. Since $w'$ begins with $01$, the transform can be applied to $w'$ in turn, and the construction iterated. To every binary word $w$ starting with $01$ and containing infinitely many $0$'s and infinitely many $1$'s, we associate its orbit by setting
\[
  \mathcal{T}^0(w)=w,\qquad
  \mathcal{T}^m(w)=\mathcal{T}(\mathcal{T}^{m-1}(w))\quad(m\ge1).
\]
This paper studies the orbit of the alternating word $w = a_0 = 010101\dots$.

\subsection*{Main results}

Write $a_m = \mathcal{T}^m(a_0)$ for the iterate of order $m$ starting from the alternating word $a_0 = 010101\dots$, and let $u_m$ and $v_m$ denote respectively the ``odious'' and ``evil'' numbers associated with $a_m$, giving the positions of the ones and the zeros in $a_m$.

The main result of this article is the explicit formula for the iterates of
$\mathcal{T}$ starting from $a_0$. From it we deduce Prouhet-Tarry-Escott identities, composition laws for the associated evil and odious numbers, uniform morphisms generating the iterates, and exact formulas for their factor complexity.

\subsubsection*{Formula for $a_m$}

Theorem~\ref{thm:masque-binaire} gives for $a_m(n)$ an explicit formula
depending on the binary expansions of $n$ and $m-1$,
\[
  a_m(n) = \bigoplus_{p\,\&\,(m-1)=0} b_p(n),
\]
which we call the binary digit formula. We give two proofs. The first
is combinatorial and uses properties of binary words. The second is
algebraic and uses the action of $\mathcal{T}$ on generating series
in $\mathbb{F}_2[[x]]$. Both proofs end with Lucas's theorem
(Lemma~\ref{lem:lucas}).

\subsubsection*{Prouhet-Tarry-Escott identities}

For every $m \ge 1$ and every $M \ge 1$,
Lemma~\ref{lem:factorisation-PTE} gives the factorization of the
polynomial $\sum_{n=0}^{2^M-1} (-1)^{a_m(n)} x^n$, from which one
deduces that the partition of $[0, 2^M)$ between the odious numbers
$u_m$ and evil numbers $v_m$ is a solution of the
Prouhet-Tarry-Escott problem. For $m=2$, this proves the Prouhet-Tarry-Escott statement suggested
by John Layman for the sequences \seqnum{A158705} and \seqnum{A158704}.

\subsubsection*{Composition laws}

Allouche, Cloitre, and Shevelev \cite{ACS2016} record the identities
\begin{align*}
  u(u(n)) &= 2u(n), & v(v(n)) &= 2v(n), \\
  u(v(n)) &= 2v(n)+1, & v(u(n)) &= 2u(n)+1
\end{align*}
for the odious numbers $u$ and evil numbers $v$ of the Thue-Morse
sequence. We extend these relations to the iterates $a_m$ for
$m \ge 1$ (Theorem~\ref{thm:compositions-evil-odious}). The four
composed sequences $x \circ y$ with $x, y \in \{u_m, v_m\}$ admit
an expression of the same type, to which is added an explicit
correction function $c_m$, which is $2$-automatic.

\subsubsection*{Factor complexity}

Theorem~\ref{thm:complexity-recurrence} gives a recurrence for the factor
complexity of every iterate. The recurrence has the same form for all
$m$, with block length
\[
  B=2^{2^{\lceil \log_2 m\rceil}}.
\]
Together with the finite seed $p_m(1),\ldots,p_m(2B-1)$, it determines
the whole function $p_m$. For the primary levels $m=2^k$,
Corollary~\ref{cor:primary-closed-complexity} gives a closed form.

\subsection*{Related work}

The construction of Prouhet-Tarry-Escott partitions from
Thue-Morse-type words goes back to Prouhet
\cite{Prouhet1851}. \v{C}ern\'y \cite{Cerny2013} studied this point
of view in terms of positions of letters in words generated by
uniform morphisms, and extended the approach to multi-dimensional
PTE solutions through composition of balanced morphisms
\cite{Cerny2017}. Bolker, Offner, Richman, and Zara
\cite{BolkerOffnerRichmanZara2016} gave other families of PTE
partitions related to generalized Thue-Morse sequences.

The present article starts from a different object. We study the
transform $\mathcal{T}$ defined by the positions of the $0$s and
the $1$s of a binary word. The PTE partitions obtained here are a
consequence of the binary digit formula for the iterates $a_m$. The
morphic structure of these iterates then places them back in the
setting of constructions by uniform morphisms, while this work
also studies the composition laws of their associated evil and
odious numbers and their factor complexity.

\subsection*{Outline}

We prove the binary digit formula in Section~\ref{sec:orbit}.
Section~\ref{sec:PTE} is devoted to the Prouhet-Tarry-Escott
problem. The composition laws for the evil and odious numbers of
the iterates of the Thue-Morse transform are the subject of
Section~\ref{sec:composition}. Section~\ref{sec:morphismes}
identifies each iterate $a_m$ as a fixed point of a uniform morphism. We use this in
Section~\ref{sec:complexity} to study the factor complexity. Some perspectives are mentioned at
the end of the article, notably a $q$-ary version of the transform
and Woods-Robbins type products attached to the iterates $a_{2^k}$.

\section{Thue-Morse orbit}
\label{sec:orbit}

We define the initial sequence by
\[
  a_0(n) = n \pmod 2,
\]
then we construct the orbit of this word by setting for every natural number $m$
\[
  a_{m+1} = \mathcal{T}(a_m).
\]
With this definition, the word $a_1$ is the Thue-Morse sequence \seqnum{A010060} and $a_2$ is the sequence \seqnum{A341389}. The following iterates $a_3$ and $a_4$ are respectively \seqnum{A395958} and \seqnum{A395961}. The first terms of the first five iterates are given in Table~\ref{tab:orbite}.

\begin{table}[H]
\centering
\small
\setlength{\tabcolsep}{3pt}
\begin{tabular}{c|cccccccccccccccccccccccccccc}
 & 0 & 1 & 2 & 3 & 4 & 5 & 6 & 7 & 8 & 9 & 10 & 11 & 12 & 13 & 14 & 15 & 16 & 17 & 18 & 19 & 20 & 21 & 22 & 23 & 24 & 25 & 26 & 27 \\
\hline
$a_{0}$ & 0 & 1 & 0 & 1 & 0 & 1 & 0 & 1 & 0 & 1 & 0 & 1 & 0 & 1 & 0 & 1 & 0 & 1 & 0 & 1 & 0 & 1 & 0 & 1 & 0 & 1 & 0 & 1 \\
$a_{1}$ & 0 & 1 & 1 & 0 & 1 & 0 & 0 & 1 & 1 & 0 & 0 & 1 & 0 & 1 & 1 & 0 & 1 & 0 & 0 & 1 & 0 & 1 & 1 & 0 & 0 & 1 & 1 & 0 \\
$a_{2}$ & 0 & 1 & 0 & 1 & 1 & 0 & 1 & 0 & 0 & 1 & 0 & 1 & 1 & 0 & 1 & 0 & 1 & 0 & 1 & 0 & 0 & 1 & 0 & 1 & 1 & 0 & 1 & 0 \\
$a_{3}$ & 0 & 1 & 1 & 0 & 0 & 1 & 1 & 0 & 0 & 1 & 1 & 0 & 0 & 1 & 1 & 0 & 1 & 0 & 0 & 1 & 1 & 0 & 0 & 1 & 1 & 0 & 0 & 1 \\
$a_{4}$ & 0 & 1 & 0 & 1 & 0 & 1 & 0 & 1 & 0 & 1 & 0 & 1 & 0 & 1 & 0 & 1 & 1 & 0 & 1 & 0 & 1 & 0 & 1 & 0 & 1 & 0 & 1 & 0 \\
\end{tabular}
\caption{The first $28$ terms of the iterates $a_0, a_1, \ldots, a_4$ of the Thue-Morse transform.}
\label{tab:orbite}
\end{table}

The columns of Table~\ref{tab:orbite} display a vertical periodicity.
For each fixed $n \ge 2$, the sequence $(a_m(n))_{m \ge 0}$ is
periodic in $m$, of period dividing $2^K$ with
$K = \lceil \log_2(\lfloor \log_2 n \rfloor + 1) \rceil$, while it is
constant for $n \in \{0,1\}$. This is a direct consequence of the binary digit formula of Theorem~\ref{thm:masque-binaire} below.

The proof uses Lucas's theorem \cite[p.~52]{Lucas1878},\cite{Granville1997}.

\begin{lemma}
\label{lem:lucas}
Let $r$ be a prime number and let $0 \le b \le a$ be integers with base-$r$ expansions $a = \sum_i a_i r^i$ and $b = \sum_i b_i r^i$. Then
\[
  \binom{a}{b} \equiv \prod_i \binom{a_i}{b_i} \pmod r,
\]
with the convention $\binom{a_i}{b_i} = 0$ when $b_i > a_i$. In particular, for $r = 2$ and $0 \le b \le a$, the binomial coefficient $\binom{a}{b}$ is odd if and only if the addition $b + (a-b)$ has no carry, equivalently $b \,\&\, (a-b) = 0$.
\end{lemma}

The following theorem gives the binary digit formula for the iterate $a_m$.

\begin{theorem}
\label{thm:masque-binaire}
For every $m \ge 1$ and every $n \ge 0$,
\[
  a_m(n) = \bigoplus_{\substack{p\ge0\\p\,\&\,(m-1)=0}} b_p(n).
\]
The XOR is finite for fixed $n$, since $b_p(n) = 0$ for every $p$ with $2^p > n$.
\end{theorem}

The proofs use the following invariant.

\begin{lemma}
\label{lem:pair-balance}
Let $w$ be a binary word in the domain of $\mathcal{T}$. If each pair $\{2k, 2k+1\}$ contains one $0$ and one $1$ in $w$, then the same property holds for $\mathcal{T}(w)$.
\end{lemma}

\begin{proof}
Assume that in $w$ each of the first $k$ pairs $\{0, 1\}, \{2, 3\}, \dots, \{2k-2, 2k-1\}$ contains one $0$ and one $1$. Then these $2k$ positions carry exactly the zeros and ones of indices $0, \dots, k-1$ of $w$. The zero and the one of index $k$ are therefore the next ones, located in the pair $\{2k, 2k+1\}$, that is,
\[
  \{v_w(k), u_w(k)\} = \{2k, 2k+1\}.
\]
The defining relations of $\mathcal{T}$ give $w'(v_w(k)) = w'(k)$ and $w'(u_w(k)) = 1 - w'(k)$, so the values of $w'$ at the positions $2k$ and $2k+1$ are complementary. The pair $\{2k, 2k+1\}$ therefore contains one $0$ and one $1$ in $w'$.
\end{proof}

The starting word $a_0 = 010101\dots$ satisfies this pair-balance property, so by Lemma~\ref{lem:pair-balance}, every iterate $a_m = \mathcal{T}^m(a_0)$ does too.

We give two proofs. The first is combinatorial. The second is algebraic. Both end with Lucas's theorem.

\subsection{Combinatorial proof}

\begin{proof}
Set $\sigma_S(n) = \bigoplus_{p \in S} b_p(n)$ for $S \subset \N$ finite and $n \ge 0$. By Lemma~\ref{lem:pair-balance} applied iteratively from $a_0$, each iterate $a_m$ has the property that the pair $\{2k, 2k+1\}$ contains one zero and one one, so $\{v_m(k), u_m(k)\} = \{2k, 2k+1\}$, where $v_m(k)$ and $u_m(k)$ denote the positions of the $k$-th zero and the $k$-th one of $a_m$. The definition of $\mathcal{T}$ gives
\[
  a_{m+1}(v_m(k)) = a_{m+1}(k), \qquad a_{m+1}(u_m(k)) = 1 - a_{m+1}(k).
\]

This property implies $v_m(k) = 2k + a_m(2k)$ and $u_m(k) = 2k + 1 - a_m(2k)$. If $a_m(2k) = 0$, then $v_m(k) = 2k$ and $u_m(k) = 2k+1$, so
\[
  a_{m+1}(2k) = a_{m+1}(k), \qquad a_{m+1}(2k+1) = 1 - a_{m+1}(k).
\]
If $a_m(2k) = 1$, the two positions are exchanged and one obtains
\[
  a_{m+1}(2k) = 1 - a_{m+1}(k), \qquad a_{m+1}(2k+1) = a_{m+1}(k).
\]
In both cases
\[
  a_{m+1}(2k) = a_{m+1}(k) \oplus a_m(2k), \qquad a_{m+1}(2k+1) = a_{m+1}(k) \oplus a_m(2k+1).
\]
These two equalities can be summarized as $a_{m+1}(n) = a_{m+1}(\lfloor n/2 \rfloor) \oplus a_m(n)$ for every $n \ge 1$, with $a_{m+1}(0) = 0$. Iterated until reaching the argument zero, this relation becomes
\[
  a_{m+1}(n) = \bigoplus_{j \ge 0} a_m(\lfloor n/2^j \rfloor),
\]
where the sum is finite since $\lfloor n/2^j \rfloor = 0$ for $j$ large enough, and $a_m(0) = 0$ for every $m$.

We use the identity
\[
  b_r(\lfloor n/2^j \rfloor) = b_{r+j}(n).
\]
For $m = 1$, starting from $a_0(n) = b_0(n)$, the above relation gives
\[
  a_1(n) = \bigoplus_{j \ge 0} b_j(n).
\]
If the formula holds at rank $m$, then
\[
a_{m+1}(n) = \bigoplus_{j_0 \ge 0} \bigoplus_{j_1, \dots, j_m \ge 0} b_{j_0 + j_1 + \cdots + j_m}(n),
\]
and one obtains by induction
\[
a_m(n) = \bigoplus_{j_1, \dots, j_m \ge 0} b_{j_1 + \cdots + j_m}(n).
\]
All these sums are finite for $n$ fixed, since $b_p(n) = 0$ as soon as $2^p > n$. The XOR of bits coincides with their integer sum modulo $2$, and for a fixed integer $p$, the number of $m$-tuples $(j_1, \dots, j_m)$ of natural numbers with $j_1 + \cdots + j_m = p$ equals $\binom{p+m-1}{m-1}$ (weak compositions of $p$ into $m$ parts). Grouping the tuples according to this sum, we obtain
\begin{equation}
a_m(n) = \sum_{p \ge 0} \binom{p+m-1}{m-1} b_p(n) \pmod 2.
\label{eq:formule-binomiale-binaire}
\end{equation}
The term indexed by $p$ contributes $b_p(n)$ modulo $2$ when $\binom{p+m-1}{m-1}$ is odd, and $0$ otherwise. Lemma~\ref{lem:lucas} shows that this parity equals $1$ if and only if the binary addition $p + (m-1)$ produces no carry, which is equivalent to $p\,\&\,(m-1) = 0$. The sum \eqref{eq:formule-binomiale-binaire} then reduces to
\[
  \bigoplus_{\substack{p \ge 0 \\ p\,\&\,(m-1) = 0}} b_p(n),
\]
which establishes the stated formula.
\end{proof}


\subsection{Algebraic proof}

The second proof of Theorem~\ref{thm:masque-binaire} uses the action of $\mathcal{T}$ on generating series in $\mathbb{F}_2[[x]]$. The binomial coefficient $\binom{p+m-1}{m-1}$ arises as the coefficient of $x^p$ in the formal expansion of $(1-x)^{-m}$ over $\mathbb{F}_2$.

For every $c \in \mathbb{F}_2^{\N}$ with $c(0) = 1$, set
\[
\sigma_c(n) = \bigoplus_{p \ge 0} c(p)\, b_p(n) \qquad (n \ge 0).
\]
For each $n$, only finitely many terms are nonzero since $b_p(n) = 0$ as soon as $2^p > n$. Let
\[
\mathcal{S} = \{\, \sigma_c : c \in \mathbb{F}_2^{\N},\ c(0) = 1 \,\}.
\]
We have $\sigma_c(0) = 0$ and $\sigma_c(2k) + \sigma_c(2k+1) = 1$ for every $k \ge 0$, so each word in $\mathcal{S}$ lies in the domain of $\mathcal{T}$. We associate to $c$ the generating series $\widehat{c}(x) = \sum_{p \ge 0} c(p)\, x^p \in \mathbb{F}_2[[x]]$. The starting word $a_0$ of the orbit lies in $\mathcal{S}$, since $a_0(n) = b_0(n) = \sigma_{\kappa_0}(n)$ with $\kappa_0(0) = 1$ and $\kappa_0(p) = 0$ for $p \ge 1$, so that $\widehat{\kappa_0}(x) = 1$.

\begin{lemma}
\label{lem:action-algebrique}
The Thue-Morse transform $\mathcal{T}$ stabilizes $\mathcal{S}$, and for every $c$ with $c(0) = 1$,
\[
  \mathcal{T}(\sigma_c) = \sigma_{c'} \qquad \text{with} \qquad c'(r) = \bigoplus_{p=0}^{r} c(p).
\]
Equivalently, on generating series in $\mathbb{F}_2[[x]]$,
\[
  \widehat{c'}(x) = (1-x)^{-1}\, \widehat{c}(x).
\]
The map $c \mapsto c'$ is a bijection on $\{c \in \mathbb{F}_2^{\N} : c(0) = 1\}$, with inverse the finite-difference operator $c''(0) = c'(0)$, $c''(r) = c'(r) \oplus c'(r-1)$ for $r \ge 1$. The transform $\mathcal{T}$ therefore acts as a bijection on $\mathcal{S}$, with inverse the difference operator on the coefficient sequence.
\end{lemma}

\begin{proof}
Let $a = \sigma_c \in \mathcal{S}$ and set $b = \mathcal{T}(a)$. Since $c(0) = 1$, the values $a(2k) = \sigma_c(2k)$ and $a(2k+1) = \sigma_c(2k+1)$ differ exactly by the bit $b_0$, so each pair $\{2k, 2k+1\}$ contains one $0$ and one $1$. Hence $\{v_a(k), u_a(k)\} = \{2k, 2k+1\}$, and the defining relations of $\mathcal{T}$ yield the recurrence
\[
  b(n) = b(\lfloor n/2 \rfloor) \oplus a(n) \qquad (n \ge 1),
\]
together with $b(0) = 0$. Iterating until the argument reaches zero, we obtain
\[
  b(n) = \bigoplus_{j \ge 0} a(\lfloor n/2^j \rfloor).
\]
Substituting $a = \sigma_c$ and using $b_p(\lfloor n/2^j \rfloor) = b_{p+j}(n)$, we exchange the order of summation,
\[
  b(n) = \bigoplus_{j \ge 0} \bigoplus_{p \ge 0} c(p)\, b_{p+j}(n) = \bigoplus_{r \ge 0} b_r(n) \bigoplus_{p = 0}^{r} c(p).
\]
We have $b = \sigma_{c'}$ with $c'(r) = \bigoplus_{p=0}^{r} c(p)$, and $c'(0) = c(0) = 1$, so $b \in \mathcal{S}$.

The translation of $c'(r) = \bigoplus_{p=0}^{r} c(p)$ into generating series is $\widehat{c'}(x) = \widehat{c}(x) / (1-x)$ in $\mathbb{F}_2[[x]]$. We thus obtain that the map $c \mapsto c'$ is multiplication by $(1-x)^{-1}$, with inverse multiplication by $(1-x)$, which on coefficients reads $c''(0) = c'(0)$ and $c''(r) = c'(r) \oplus c'(r-1)$ for $r \ge 1$. Both maps preserve the condition $c(0) = 1$, so $\mathcal{T}$ is a bijection on $\mathcal{S}$.
\end{proof}

\begin{remark}
On the coefficient sequence, the action of $\mathcal{T}$ corresponds to the infinite lower-triangular matrix $P$ over $\mathbb{F}_2$ with $P_{r,p} = 1$ for $0 \le p \le r$ and $0$ otherwise,
\[
P =
\begin{pmatrix}
1 & 0 & 0 & 0 & 0 & \cdots \\
1 & 1 & 0 & 0 & 0 & \cdots \\
1 & 1 & 1 & 0 & 0 & \cdots \\
1 & 1 & 1 & 1 & 0 & \cdots \\
1 & 1 & 1 & 1 & 1 & \cdots \\
\vdots & \vdots & \vdots & \vdots & \vdots & \ddots
\end{pmatrix}.
\]
Its inverse $P^{-1}$ has entries $1$ on the diagonal and on the subdiagonal, and $0$ elsewhere,
\[
P^{-1} =
\begin{pmatrix}
1 & 0 & 0 & 0 & 0 & \cdots \\
1 & 1 & 0 & 0 & 0 & \cdots \\
0 & 1 & 1 & 0 & 0 & \cdots \\
0 & 0 & 1 & 1 & 0 & \cdots \\
0 & 0 & 0 & 1 & 1 & \cdots \\
\vdots & \vdots & \vdots & \vdots & \vdots & \ddots
\end{pmatrix}.
\]
The $m$-th power $P^m$ has entries $(P^m)_{r,p} = \binom{r-p+m-1}{m-1} \pmod 2$ for $0 \le p \le r$ and $0$ otherwise. The matrix $P^m$ is therefore Toeplitz, with the sequence $\binom{p+m-1}{m-1} \pmod 2$ as its first column. Since $\kappa_m = P^m \kappa_0$ with $\kappa_0 = (1, 0, 0, \dots)^T$, this first column is exactly the coefficient sequence $\kappa_m$ appearing in the binary digit formula. For instance, $m = 3$ gives
\[
P^3 =
\begin{pmatrix}
1 & 0 & 0 & 0 & 0 & 0 & 0 & \cdots \\
1 & 1 & 0 & 0 & 0 & 0 & 0 & \cdots \\
0 & 1 & 1 & 0 & 0 & 0 & 0 & \cdots \\
0 & 0 & 1 & 1 & 0 & 0 & 0 & \cdots \\
1 & 0 & 0 & 1 & 1 & 0 & 0 & \cdots \\
1 & 1 & 0 & 0 & 1 & 1 & 0 & \cdots \\
0 & 1 & 1 & 0 & 0 & 1 & 1 & \cdots \\
\vdots & \vdots & \vdots & \vdots & \vdots & \vdots & \vdots & \ddots
\end{pmatrix},
\]
whose first column $(1, 1, 0, 0, 1, 1, 0, \dots)$ matches $\kappa_3(p) = \binom{p+2}{2} \pmod 2$.
\end{remark}

The second proof of the binary digit formula follows by iterated application of Lemma~\ref{lem:action-algebrique}.

\begin{proof}[Second proof of Theorem~\ref{thm:masque-binaire}]
Since $a_0 = \sigma_{\kappa_0}$ with $\widehat{\kappa_0}(x) = 1$, Lemma~\ref{lem:action-algebrique} applied $m$ times gives $a_m = \sigma_{\kappa_m}$ with
\[
  \widehat{\kappa_m}(x) = (1-x)^{-m} \qquad \text{in } \mathbb{F}_2[[x]].
\]
The formal expansion of $(1-x)^{-m}$ yields
\[
  \kappa_m(p) = \binom{p+m-1}{m-1} \pmod 2 \qquad (p \ge 0),
\]
hence
\[
  a_m(n) = \sigma_{\kappa_m}(n) = \bigoplus_{p \ge 0} \binom{p+m-1}{m-1} b_p(n) \pmod 2.
\]
By Lemma~\ref{lem:lucas}, $\binom{p+m-1}{m-1}$ is odd if and only if the binary addition $p + (m-1)$ produces no carry, that is, $p \,\&\, (m-1) = 0$. The sum above therefore reduces to
\[
  a_m(n) = \bigoplus_{\substack{p \ge 0 \\ p\,\&\,(m-1) = 0}} b_p(n),
\]
which is the formula stated in Theorem~\ref{thm:masque-binaire}.
\end{proof}

\section{Prouhet-Tarry-Escott identities}
\label{sec:PTE}

The binary digit formula transforms each iterate $a_m$ into a partition defined by a parity condition on the binary digits of $n$, from which one derives PTE identities. We use the standard fact that the order of vanishing at $x=1$ of the signed generating polynomial of a partition controls its PTE degree.

\begin{lemma}
\label{lem:factorisation-PTE}
For every integer $M \ge 1$ and every subset $S \subset \{0, 1, \dots, M-1\}$, setting $f_S(n) = \bigoplus_{p \in S} b_p(n)$, one has the identity in $\Z[x]$
\[
  \sum_{n=0}^{2^M-1} (-1)^{f_S(n)} x^n = \prod_{p \in S} (1 - x^{2^p}) \prod_{\substack{0 \le p < M \\ p \notin S}} (1 + x^{2^p}).
\]
The order of vanishing of the left-hand side at $x = 1$ equals $|S|$. If $S \neq \emptyset$, this provides a solution to the Prouhet-Tarry-Escott problem of degree $|S| - 1$.
\end{lemma}

\begin{proof}
Every integer $n \in \{0,\dots,2^M-1\}$ has a unique binary expansion $n = \sum_{p=0}^{M-1} b_p(n) 2^p$ with $b_p(n) \in \{0,1\}$, and $(-1)^{f_S(n)} = \prod_{p \in S} (-1)^{b_p(n)}$. Set $\varepsilon_p = -1$ for $p \in S$ and $\varepsilon_p = +1$ otherwise. The expansion
\[
  \prod_{p=0}^{M-1} (1 + \varepsilon_p x^{2^p}) = \sum_{(\beta_0,\dots,\beta_{M-1}) \in \{0,1\}^M} \left(\prod_{p} \varepsilon_p^{\beta_p}\right) x^{\sum_p \beta_p 2^p}
\]
runs over integers $n \in \{0,\dots,2^M-1\}$ bijectively via $\beta_p = b_p(n)$. Since $\varepsilon_p = +1$ for $p \notin S$, the coefficient of $x^n$ is $\prod_{p \in S} (-1)^{b_p(n)} = (-1)^{f_S(n)}$. This proves the identity.

Each factor $1 - x^{2^p}$ vanishes at $x=1$ with multiplicity one, so the total order of vanishing is $|S|$. Let $P(x) = \sum_{n=0}^{2^M-1} (-1)^{f_S(n)} x^n$. The $k$-th derivative evaluated at $x = 1$ gives $P^{(k)}(1) = \sum_{n} (-1)^{f_S(n)} n(n-1)\cdots(n-k+1)$, the signed factorial moment of order $k$. The vanishing of $P$ to order $|S|$ at $x = 1$ implies $P^{(k)}(1) = 0$ for $0 \le k \le |S|-1$, hence the equality of factorial moments of order $0$ to $|S|-1$ between the sets $\{n : f_S(n) = 0\}$ and $\{n : f_S(n) = 1\}$. Since the powers $n^k$ form triangular linear combinations of the factorial moments $n^{(j)} = n(n-1)\cdots(n-j+1)$ for $j \le k$, the equality of factorial moments up to order $|S|-1$ is equivalent to the equality of power sums $\sum n^k$ up to degree $|S|-1$.
\end{proof}

For the iterate $a_m$, the set $S$ is defined by the condition $p\,\&\,(m-1)=0$ in the block $[0, 2^M)$.

\begin{corollary}
\label{cor:PTE-am}
For every $m \ge 1$ and every $M \ge 1$, set $S_{m,M} = \{0 \le p < M : p\,\&\,(m-1) = 0\}$. The partition of $[0, 2^M)$ between the generalized evil and odious numbers associated with $a_m$ provides a solution to the Prouhet-Tarry-Escott problem of degree $|S_{m,M}| - 1$.
\end{corollary}

\begin{proof}
We apply Lemma~\ref{lem:factorisation-PTE} with $S = S_{m,M}$. The order of vanishing at $x = 1$ equals $|S_{m,M}|$ and yields the equality of power sums up to degree $|S_{m,M}| - 1$ between the positions of the $0$s and those of the $1$s of $a_m$ in the block.
\end{proof}

In particular, the case $m = 2$ answers Layman's question.

\begin{corollary}
\label{cor:PTE-Layman}
For $m = 2$ and every $M \ge 1$, $S_{2,M} = \{0 \le p < M : p \text{ even}\}$, hence $|S_{2,M}| = \lceil M/2 \rceil$. The partition of $[0, 2^M)$ between the sequences \seqnum{A158704} and \seqnum{A158705} provides a solution to the Prouhet-Tarry-Escott problem of degree $\lceil M/2 \rceil - 1$.
\end{corollary}

Iterating the transform thus yields an infinite family of Prouhet-Tarry-Escott identities. By Corollary~\ref{cor:PTE-am}, each level $m \ge 1$ gives a partition of $[0, 2^M)$ into two classes balanced in degree $|S_{m,M}|-1$. At level $m = 3$, these two classes are the sequences \seqnum{A395959} (positions of the $0$s of $a_3$) and \seqnum{A395960} (positions of the $1$s). At level $m = 4$, they are \seqnum{A395962} and \seqnum{A395963}.
\section{Generalized evil and odious numbers}
\label{sec:composition}

For each iterate $m \ge 1$, we denote by $u_m(n)$ the position of the $n$-th $1$ of $a_m$ and by $v_m(n)$ the position of the $n$-th $0$ of $a_m$.

The binary digit formula implies that each pair $\{2n,2n+1\}$ contains exactly one $0$ and one $1$ of the word $a_m$. Indeed, the bit $b_0$ always contributes to the formula, since $0\,\&\,(m-1)=0$. We thus obtain the following formulas.

\begin{lemma}
\label{lem:pairing-uv}
For every $m\ge 1$ and every $n\ge 0$,
\[
  u_m(n)=2n+1-a_m(2n),
  \qquad
  v_m(n)=2n+a_m(2n).
\]
\end{lemma}

\begin{proof}
Since $0\,\&\,(m-1)=0$, the two terms $a_m(2n)$ and $a_m(2n+1)$ differ exactly by the bit $b_0$. We thus have
\[
  a_m(2n+1)=1-a_m(2n).
\]
Thus, in the pair $\{2n,2n+1\}$, one of the two positions carries a $0$ and the other a $1$. If $a_m(2n)=0$, then $v_m(n)=2n$ and $u_m(n)=2n+1$. Otherwise, the two roles are exchanged. The stated formulas follow in both cases.
\end{proof}

We use these relations to express the compositions of $u_m$ and $v_m$ via $2$-automatic correction functions $c_m$. Table~\ref{tab:correction} provides examples.

\begin{definition}
We define, for each $m\ge 1$, a periodic subset $C_m \subset \N$ and a correction function $c_m$.

We first set
\[
  C_1=\emptyset,
  \qquad
  c_1(n)=0.
\]

If $m\ge 2$ is even, we set
\[
  K=\lceil \log_2 m\rceil, \qquad P=2^K,
\]
and $C_m$ is the periodic subset of period $P$ defined by
\[
  C_m\cap[0,P)
  =
  \{\,q\in\{0,1,\dots,P-1\} : (q+2)\,\&\,(m-1)=0\,\}.
\]

If $m\ge 3$ is odd, we set
\[
  C_m=C_{m-1}.
\]

In all cases, the correction function is
\[
  c_m(n)=\bigoplus_{q\in C_m} b_q(n).
\]
The sum is finite for each $n$, since only finitely many bits $b_q(n)$ are nonzero. In particular, $c_m(n) \in \{0, 1\}$ for every $m \ge 1$ and every $n \ge 0$.
\end{definition}

\begin{example}
The first correction sets are
\[
  C_1=\emptyset,
\]
\[
  C_2=C_3=\{0,2,4,6,\dots\},
\]
\[
  C_4=C_5=\{2,6,10,14,\dots\},
\]
and
\[
  C_6=C_7=\{0,6,8,14,16,22,\dots\}.
\]
\end{example}

Table~\ref{tab:correction} shows the first $28$ values of $c_1, c_2, c_3, c_4$.

\begin{table}[H]
\centering
\small
\setlength{\tabcolsep}{3pt}
\begin{tabular}{c|cccccccccccccccccccccccccccc}
 & 0 & 1 & 2 & 3 & 4 & 5 & 6 & 7 & 8 & 9 & 10 & 11 & 12 & 13 & 14 & 15 & 16 & 17 & 18 & 19 & 20 & 21 & 22 & 23 & 24 & 25 & 26 & 27 \\
\hline
$c_{1}$ & 0 & 0 & 0 & 0 & 0 & 0 & 0 & 0 & 0 & 0 & 0 & 0 & 0 & 0 & 0 & 0 & 0 & 0 & 0 & 0 & 0 & 0 & 0 & 0 & 0 & 0 & 0 & 0 \\
$c_{2}$ & 0 & 1 & 0 & 1 & 1 & 0 & 1 & 0 & 0 & 1 & 0 & 1 & 1 & 0 & 1 & 0 & 1 & 0 & 1 & 0 & 0 & 1 & 0 & 1 & 1 & 0 & 1 & 0 \\
$c_{3}$ & 0 & 1 & 0 & 1 & 1 & 0 & 1 & 0 & 0 & 1 & 0 & 1 & 1 & 0 & 1 & 0 & 1 & 0 & 1 & 0 & 0 & 1 & 0 & 1 & 1 & 0 & 1 & 0 \\
$c_{4}$ & 0 & 0 & 0 & 0 & 1 & 1 & 1 & 1 & 0 & 0 & 0 & 0 & 1 & 1 & 1 & 1 & 0 & 0 & 0 & 0 & 1 & 1 & 1 & 1 & 0 & 0 & 0 & 0 \\
\end{tabular}
\caption{The first $28$ terms of the correction functions $c_1, c_2, c_3, c_4$.}
\label{tab:correction}
\end{table}

We require two lemmas. The first controls the value of $a_m$ when going from $4n$ to $4n+2$.

\begin{lemma}
\label{lem:am-4n2}
For every $m\ge 1$ and every $n\ge 0$,
\[
  a_m(4n+2)=a_m(4n)\oplus \varepsilon_m,
\]
where
\[
  \varepsilon_m=
  \begin{cases}
  1, & \text{if } m \text{ is odd},\\
  0, & \text{if } m \text{ is even}.
  \end{cases}
\]
\end{lemma}

\begin{proof}
The integers $4n$ and $4n+2$ differ only by the bit at position $1$, that is, $b_1(4n) \neq b_1(4n+2)$ while $b_p(4n) = b_p(4n+2)$ for all $p \neq 1$. By the binary digit formula, $a_m(4n) \neq a_m(4n+2)$ if and only if the term $b_1$ contributes to the formula, that is, if and only if
\[
  1\,\&\,(m-1)=0.
\]
This is equivalent to $m-1$ being even, or equivalently $m$ being odd.
\end{proof}

The second lemma is a bitwise identity used only in the odd case.

\begin{lemma}
\label{lem:mask-even-shift}
Let $M\ge 2$ be an even integer. For every integer $x\ge 1$,
\[
\mathbf{1}_{x\,\&\,M=0}
\oplus
\mathbf{1}_{(x-1)\,\&\,M=0}
=
\mathbf{1}_{x\,\&\,(M-1)=0}.
\]
\end{lemma}

\begin{proof}
Write $M = 2^t s$ with $t \ge 1$ and $s$ odd.

If $x$ is not divisible by $2^t$, write $x = 2^t Q + r$ with $1 \le r < 2^t$. Then $x-1 = 2^t Q + (r-1)$, so $x$ and $x-1$ have the same quotient $Q$ in the division by $2^t$. The two indicators on the left-hand side are then equal, and their XOR is $0$. Moreover, one of the $t$ low-order bits of $x$ equals $1$. Since $M-1$ has all its $t$ low-order bits equal to $1$, we have $x \,\&\, (M-1) \ne 0$, and the right-hand side is also $0$.

Suppose now that $x = 2^t X$. We have
\[
  x \,\&\, M = 2^t (X \,\&\, s).
\]
Moreover
\[
  x - 1 = 2^t (X - 1) + (2^t - 1),
\]
and $M = 2^t s$ has all its $t$ lowest bits zero. Thus
\[
  (x-1) \,\&\, M = 2^t ((X-1) \,\&\, s).
\]
We obtain the two indicator equalities
\[
  \mathbf{1}_{x \,\&\, M = 0} = \mathbf{1}_{X \,\&\, s = 0}, \qquad \mathbf{1}_{(x-1) \,\&\, M = 0} = \mathbf{1}_{(X-1) \,\&\, s = 0}.
\]
Write $s = 1 + 2q$. If $X = 2Y$, then
\[
  X \,\&\, s = 2(Y \,\&\, q),
\]
while $X-1$ and $s$ are odd, so $(X-1) \,\&\, s \ne 0$. The left-hand side then equals $\mathbf{1}_{Y \,\&\, q = 0}$. If $X = 2Y + 1$, then $X$ and $s$ are odd, so $X \,\&\, s \ne 0$, while
\[
  (X-1) \,\&\, s = (2Y) \,\&\, (1 + 2q) = 2(Y \,\&\, q).
\]
The left-hand side again equals $\mathbf{1}_{Y \,\&\, q = 0}$. For the right-hand side, we have
\[
  M - 1 = 2^t (s - 1) + (2^t - 1),
\]
so $x \,\&\, (M-1) = 0$ is equivalent to
\[
  X \,\&\, (s-1) = 0.
\]
Since the lowest-order bit of $s-1 = 2q$ is zero, $X \,\&\, 2q$ depends only on the bits of position $\ge 1$ of $X$, that is, on $\lfloor X/2 \rfloor$. This condition is therefore equivalent to
\[
  \lfloor X/2 \rfloor \,\&\, q = 0.
\]
This proves the identity.
\end{proof}

The functions $c_m$ can be read directly from $a_m$.

\begin{lemma}
\label{lem:correction-as-values}
For every $n\ge 0$, we have the following identities.

If $m$ is even, then
\[
  c_m(n)=a_m(4n).
\]

If $m$ is odd, then
\[
  c_m(n)=a_m(4n)\oplus a_m(2n).
\]
\end{lemma}

\begin{proof}
Suppose first that $m$ is even. Then $m-1$ is odd. Multiplication by $4$ shifts the binary expansion by two positions, so $b_0(4n) = b_1(4n) = 0$ and $b_p(4n) = b_{p-2}(n)$ for $p \ge 2$. The binary digit formula therefore gives
\[
  a_m(4n) = \bigoplus_{\substack{p \ge 2 \\ p \,\&\, (m-1) = 0}} b_{p-2}(n) = \bigoplus_{\substack{q \ge 0 \\ (q+2) \,\&\, (m-1) = 0}} b_q(n).
\]
Since $m-1 < P$, the condition $(q+2) \,\&\, (m-1) = 0$ depends only on the $K$ low-order bits of $q+2$. It is therefore periodic in $q$ with period $P$. By the definition of $C_m$, the set of indices $q$ that occur is exactly $C_m$, and the right-hand side equals $c_m(n)$.

Suppose now that $m$ is odd. The case $m = 1$ is immediate. Indeed, $c_1(n) = 0$, and the Thue-Morse word satisfies $a_1(2r) = a_1(r)$. Taking $r = 2n$ and then $r = n$, we obtain $a_1(4n) = a_1(2n)$.

We can therefore assume $m \ge 3$. Then $m-1$ is even and, by definition, $C_m = C_{m-1}$, so $c_m = c_{m-1}$. Since $m-1$ is even, the case already treated gives $c_{m-1}(n) = a_{m-1}(4n)$. We verify that
\[
  a_m(4n) \oplus a_m(2n) = a_{m-1}(4n).
\]
Applying the binary digit formula to $4n$ and to $2n$, and using the relations $b_p(4n) = b_{p-2}(n)$ for $p \ge 2$ and $b_p(2n) = b_{p-1}(n)$ for $p \ge 1$, we have
\[
  a_m(4n) = \bigoplus_{\substack{q \ge 0 \\ (q+2) \,\&\, (m-1) = 0}} b_q(n),
\]
and
\[
  a_m(2n) = \bigoplus_{\substack{q \ge 0 \\ (q+1) \,\&\, (m-1) = 0}} b_q(n).
\]
In the XOR $a_m(4n) \oplus a_m(2n)$, the term $b_q(n)$ appears with coefficient
\[
  \mathbf{1}_{(q+2) \,\&\, (m-1) = 0} \oplus \mathbf{1}_{(q+1) \,\&\, (m-1) = 0}.
\]
Since $m \ge 3$ is odd, $M = m-1$ is even with $M \ge 2$, and $x = q+2 \ge 2$, so Lemma~\ref{lem:mask-even-shift} applies. The coefficient is therefore
\[
  \mathbf{1}_{(q+2) \,\&\, (m-2) = 0}.
\]
Thus
\[
  a_m(4n) \oplus a_m(2n) = \bigoplus_{\substack{q \ge 0 \\ (q+2) \,\&\, (m-2) = 0}} b_q(n).
\]
By the even case already established, applied to $m-1$, this sum equals $a_{m-1}(4n)$. Since $c_m(n) = c_{m-1}(n) = a_{m-1}(4n)$, the result follows.
\end{proof}

We can now establish the composition laws.

\begin{theorem}
\label{thm:compositions-evil-odious}
For \(\alpha\in\{0,1\}\), set
\[
  w_m^0=v_m,\qquad w_m^1=u_m.
\]
Let
\[
  \varepsilon_m \equiv m \pmod 2,\qquad \varepsilon_m\in\{0,1\}.
\]
Then, for every \(m\ge1\), every \(n\ge0\), and every
\(\alpha,\beta\in\{0,1\}\),
\[
  w_m^\alpha\bigl(w_m^\beta(n)\bigr)
  =
  2w_m^\beta(n)
  +
  \bigl(c_m(n)\oplus \alpha\oplus \varepsilon_m\beta\bigr).
\]
The last term is read as an element of \(\{0,1\}\).
\end{theorem}

\begin{proof}
Set
\[
  \epsilon=a_m(2n).
\]
By Lemma~\ref{lem:pairing-uv},
\[
  w_m^\beta(n)=2n+(\epsilon\oplus\beta).
\]
Applying the same lemma again gives
\[
  w_m^\alpha\bigl(w_m^\beta(n)\bigr)
  =2w_m^\beta(n)
   +\bigl(a_m(2w_m^\beta(n))\oplus\alpha\bigr).
\]
It remains to compute \(a_m(2w_m^\beta(n))\).

If \(m\) is even, then Lemma~\ref{lem:am-4n2} gives
\[
  a_m(4n+2)=a_m(4n).
\]
Since \(2w_m^\beta(n)\) is either \(4n\) or \(4n+2\), we get
\[
  a_m(2w_m^\beta(n))=a_m(4n)=c_m(n)
\]
by Lemma~\ref{lem:correction-as-values}. This is
\(c_m(n)\oplus \varepsilon_m\beta\), since \(\varepsilon_m=0\).

If \(m\) is odd, write \(\gamma=a_m(4n)\). Lemma~\ref{lem:am-4n2}
gives
\[
  a_m(4n+2)=\gamma\oplus1.
\]
Also, by Lemma~\ref{lem:correction-as-values},
\[
  c_m(n)=a_m(4n)\oplus a_m(2n)=\gamma\oplus\epsilon.
\]
Since
\[
  2w_m^\beta(n)=4n+2(\epsilon\oplus\beta),
\]
we have
\[
  a_m(2w_m^\beta(n))
  =\gamma\oplus\epsilon\oplus\beta
  =c_m(n)\oplus\beta.
\]
This is again \(c_m(n)\oplus\varepsilon_m\beta\), since
\(\varepsilon_m=1\).

In both cases,
\[
  a_m(2w_m^\beta(n))=c_m(n)\oplus\varepsilon_m\beta.
\]
Substitution gives the stated formula.
\end{proof}

Writing this out gives the following identities. If \(m\) is even, then
\[
\begin{aligned}
  u_m(u_m(n)) &= 2u_m(n)+1-c_m(n),\\
  v_m(v_m(n)) &= 2v_m(n)+c_m(n),\\
  u_m(v_m(n)) &= 2v_m(n)+1-c_m(n),\\
  v_m(u_m(n)) &= 2u_m(n)+c_m(n).
\end{aligned}
\]
If \(m\) is odd, then
\[
\begin{aligned}
  u_m(u_m(n)) &= 2u_m(n)+c_m(n),\\
  v_m(v_m(n)) &= 2v_m(n)+c_m(n),\\
  u_m(v_m(n)) &= 2v_m(n)+1-c_m(n),\\
  v_m(u_m(n)) &= 2u_m(n)+1-c_m(n).
\end{aligned}
\]
\section{Morphic structure of the iterates}
\label{sec:morphismes}
\label{sec:morphic}

The binary digit formula of Theorem~\ref{thm:masque-binaire} shows that each iterate $a_m$ is $2$-automatic; see \cite{AlloucheShallit2003} for background. The uniform morphism generating $a_m$ is as follows.

For a level $m \ge 1$, set $k = \lceil \log_2 m \rceil$ and $B = 2^{2^k}$. Every integer $N \ge 0$ decomposes uniquely as $N = Bn + r$ with $0 \le r < B$. We refer to the segment $[Bn, Bn + B)$ as the block of index $n$ in $a_m$.
The condition $p\,\&\,(m-1)=0$ depends only on $p$ modulo $2^k$, since $m-1<2^k$. Writing $h=2^k$ and $B=2^h$, the decomposition $N=Bn+r$ separates the low $h$ binary digits, carried by $r$, from the remaining digits, carried by $n$. Hence the condition $p\,\&\,(m-1)=0$ splits as
\[
  a_m(Bn+r)=a_m(n)\oplus \lambda_m(r),
\]
where $\lambda_m(r) = \bigoplus_{p\,\&\,(m-1)=0} b_p(r)$ is the local pattern on the block.

This relation implies the following result.

\begin{proposition}
For every $m \ge 1$, setting $k = \lceil \log_2 m \rceil$ and $B = 2^{2^k}$, the word $a_m$ is the fixed point starting with $0$ of the uniform morphism $\mu_m$ of length $B$ defined on $\{0,1\}$ by
\begin{align*}
  \mu_m(0) &= \lambda_m(0)\lambda_m(1)\dots\lambda_m(B-1), \\
  \mu_m(1) &= \overline{\lambda_m(0)}\,\overline{\lambda_m(1)}\dots\overline{\lambda_m(B-1)},
\end{align*}
where
\[
  \lambda_m(r)=\bigoplus_{p\,\&\,(m-1)=0} b_p(r),
\]
and \(\overline{0}=1,\ \overline{1}=0\).
\end{proposition}

\begin{proof}
We derive the relation $a_m(Bn + r) = a_m(n) \oplus \lambda_m(r)$ for $0 \le r < B$ from the binary digit formula. Write $Bn + r$ in base $2$. The bits at positions $0, \dots, h - 1$ encode $r$ and the bits at positions $h, h+1, \dots$ encode $n$. Split the indices $p$ in the binary digit formula into the ranges $p < h$ and $p \ge h$. The first range contributes
\[
  \bigoplus_{\substack{p < h \\ p\,\&\,(m-1) = 0}} b_p(Bn + r) = \bigoplus_{\substack{p < h \\ p\,\&\,(m-1) = 0}} b_p(r) = \lambda_m(r).
\]
For $p = h + q$ with $q \ge 0$, we have $p \,\&\, (m-1) = q \,\&\, (m-1)$ since $m - 1 < h$, and $b_{h + q}(Bn + r) = b_q(n)$. The second range therefore contributes $a_m(n)$. Summing gives $a_m(Bn + r) = a_m(n) \oplus \lambda_m(r)$.

The block of length $B$ starting at position $Bn$ in $a_m$ is the pattern $\lambda_m(0) \dots \lambda_m(B-1)$ when $a_m(n) = 0$, and its complement when $a_m(n) = 1$. Since $a_m(0) = 0$, the word $a_m$ is a fixed point of $\mu_m$ starting with $0$. The morphism $\mu_m$ is prolongable on $0$, since $\mu_m(0)$ begins with $\lambda_m(0) = 0$. Hence $\lim_{t \to \infty} \mu_m^t(0)$ is the unique fixed point of $\mu_m$ beginning with $0$, and coincides with $a_m$.
\end{proof}

We examine this structure for the first levels.

\begin{example}
For $m = 2^k$, the binary digit formula selects only the bit at position $p=0$ within the remainder $r$, that is, its parity. The local pattern is therefore a strict alternation $\lambda_m(r) = r \bmod 2$.
\begin{itemize}
    \item For $m=1$ ($k=0, B=2$), we recover the Thue-Morse morphism, with $\mu_1(0) = 01$ and $\mu_1(1) = 10$.
    \item For $m=2$ ($k=1, B=4$), we obtain the Layman sequence morphism, with $\mu_2(0) = 0101$ and $\mu_2(1) = 1010$.
    \item For $m=4$ ($k=2, B=16$), the morphism extends over 16 letters, with $\mu_4(0) = 0101010101010101$ and $\mu_4(1) = 1010101010101010$.
\end{itemize}
\end{example}

\begin{remark}
For $m = 2^k$, the relation $a_{2^k}(Bn+r) = a_{2^k}(n) \oplus (r \bmod 2)$ shows that each letter of $a_{2^k}$ is replaced by an alternating block of length $B = 2^{2^k}$, or by its complement. If $m$ is not a power of two, then the binary digit formula has several active bit positions below $2^k$. The block pattern $\lambda_m$ is then no longer a strict alternation. This is the case treated in Section~\ref{subsec:reduction-generale}.
\end{remark}

\begin{example}
For the iterate $a_3$ ($m=3$), we have $k=2$, $B=16$, and $m-1 = 10_2$,
so the binary digit formula reduces to $\lambda_3(r) = b_0(r) \oplus
b_1(r)$. Evaluated on $r \in [0, 15]$, it generates the repetition of
the pattern $0110$. The morphism of length 16 is then given by
\begin{align*}
  \mu_3(0) &= 0110011001100110 \\
  \mu_3(1) &= 1001100110011001
\end{align*}
\end{example}

\begin{remark}
For every $m \ge 1$, the morphism $\mu_m$ is balanced in the sense of \v{C}ern\'y \cite{Cerny2013, Cerny2017}. More precisely, set $s_m = \#\{0 \le p < h : p\,\&\,(m-1) = 0\}$. The local signed polynomial
\[
\sum_{r=0}^{B-1} (-1)^{\lambda_m(r)}\, x^r \;=\; \prod_{\kappa_m(p) = 1} (1 - x^{2^p}) \prod_{\kappa_m(p) = 0} (1 + x^{2^p})
\]
factors with a zero of order $s_m$ at $x = 1$, so $\mu_m(0)$ and $\mu_m(1)$ are $s_m$-PTE equivalent in the sense of \v{C}ern\'y \cite{Cerny2013}. Iterating gives the PTE identities of Section~\ref{sec:PTE} on the blocks $[0, B^L)$ for every $L \ge 1$. The factorization argument used in Section~\ref{sec:PTE} covers in addition the blocks $[0, 2^M)$ for $M$ not a multiple of $h$.
\end{remark}

We use this morphic structure to study the factor complexity of $a_m$ in Section~\ref{sec:complexity}.

\section{Factor complexity}
\label{sec:complexity}

The factor complexity function \(p_m(n)\) counts the number of distinct
factors of length \(n\) in the word \(a_m\). This section gives an effective
description of \(p_m\) for every \(m\ge1\). The general part is the
recurrence. It reduces the computation of \(p_m\) to a finite seed. For
the primary levels, this seed has a closed form.

For \(m=1\), this is the classical complexity of the Thue-Morse word. Its
successive values are \seqnum{A005942}, studied independently by
Brlek~\cite{Brlek1989} and by de Luca and
Varricchio~\cite{deLucaVarricchio1989}.

\subsection{The recurrence}
\label{subsec:complexity-recurrence}

The main result of this subsection is the following recurrence.

\begin{theorem}
\label{thm:complexity-recurrence}
Let \(m\ge1\), and set
\[
  k=\lceil\log_2 m\rceil,\qquad B=2^{2^k}.
\]
Then, for every \(n\ge2\),
\begin{equation}
\label{eq:univ-r0}
  p_m(Bn)=p_m(n)+(B-1)p_m(n+1),
\end{equation}
and, for every \(n\ge2\) and \(1\le r\le B-1\),
\begin{equation}
\label{eq:univ-rne0}
  p_m(Bn+r)=(B+1-r)p_m(n+1)+(r-1)p_m(n+2).
\end{equation}
\end{theorem}

For example, when \(m=1\), we have \(B=2\), and the recurrence becomes
Brlek's recurrence for the Thue-Morse word:
\[
  p_1(2n)=p_1(n)+p_1(n+1),\qquad
  p_1(2n+1)=2p_1(n+1).
\]
When \(m=2\), we have \(B=4\), and the recurrence is
\begin{align*}
  p_2(4n) &= p_2(n) + 3p_2(n+1), \\
  p_2(4n+1) &= 4p_2(n+1), \\
  p_2(4n+2) &= 3p_2(n+1) + p_2(n+2), \\
  p_2(4n+3) &= 2p_2(n+1) + 2p_2(n+2).
\end{align*}

The proof uses the derivative word
\[
  \Delta_m(n)=a_m(n)\oplus a_m(n+1).
\]
First, Lemma~\ref{lem:complexity-derivative-general} gives
\[
  p_m(n)=2q_m(n-1),
\]
where \(q_m\) is the factor complexity of \(\Delta_m\). Then a
synchronization property for \(\Delta_m\) gives the recurrence.

\subsubsection{Derivative and reduction}

\label{subsec:reduction-generale}

We show here that the validity of the system \eqref{eq:univ-r0}-\eqref{eq:univ-rne0} for a level $m$ reduces to a phase synchronization property of the derivative sequence.

Fix $m \ge 1$ with $M = m-1$, $k = \lceil \log_2 m \rceil$, $h = 2^k$ and $B = 2^h$. The condition $p\,\&\,M=0$ depends only on $p \bmod h$. Define the set of positions read in the first block of size $h$ by
\[
R_m = \{0 \le s < h : s\,\&\,M = 0\}.
\]
For every remainder $0 \le r < B$, set $\lambda_m(r) = \bigoplus_{s \in R_m} b_s(r)$. The Euclidean decomposition $N = Bn+r = 2^h n + r$ then yields
\[
a_m(Bn+r) = a_m(n) \oplus \lambda_m(r).
\]
We introduce the derivative sequence $\Delta_m(n) = a_m(n) \oplus a_m(n+1)$, which has an explicit formula.

\begin{lemma}\label{lem:derivative-formula}
The derivative sequence satisfies $\Delta_m(n) = \gamma_m(\nu_2(n+1))$,
where $\nu_2$ denotes the $2$-adic valuation and
\[
\gamma_m(t) = \#\{0 \le s \le t : s\,\&\,M = 0\} \pmod 2.
\]
\end{lemma}

\begin{proof}
Set $t = \nu_2(n+1)$. Then the bits $0, \dots, t-1$ of $n$ are equal to $1$, while bit $t$ is $0$. Adding $1$, the bits $0, \dots, t$ are exactly the bits that change.

By the binary digit formula,
\[
  a_m(n) = \bigoplus_{\substack{s \ge 0 \\ s\,\&\,M = 0}} b_s(n).
\]
In the XOR $a_m(n) \oplus a_m(n+1)$, all bits that do not change cancel. What remains is exactly the parity of the number of positions $s \in \{0, \dots, t\}$ such that $s\,\&\,M = 0$. We thus obtain
\[
  \Delta_m(n) = \gamma_m(t) = \gamma_m(\nu_2(n+1)).
\]
\end{proof}

We isolate the block structure of $\Delta_m$ in the following lemma.

\begin{lemma}
\label{lem:block-form-derivative}
Let $m \ge 1$, $M = m - 1$, $k = \lceil \log_2 m \rceil$, $h = 2^k$, $B = 2^h$, and
\[
  R_m = \{0 \le s < h : s \,\&\, M = 0\}.
\]
Define
\[
  \theta_m(r) = \gamma_m(\nu_2(r + 1)) \qquad (0 \le r \le B - 2),
\]
and
\[
  \eta_m = \#R_m \pmod 2.
\]
Then, for every $q \ge 0$,
\[
  \Delta_m(qB + r) = \theta_m(r) \qquad (0 \le r \le B - 2),
\]
and
\[
  \Delta_m(qB + B - 1) = \eta_m \oplus \Delta_m(q).
\]
\end{lemma}

\begin{proof}
For $0 \le r \le B - 2$, we have $r + 1 \le B - 1 < 2^h$, hence $\nu_2(qB + r + 1) = \nu_2(r + 1)$. Lemma~\ref{lem:derivative-formula} gives
\[
  \Delta_m(qB + r) = \gamma_m(\nu_2(r + 1)) = \theta_m(r).
\]
For $r = B - 1$, we have $qB + B = (q + 1) B$, so $\nu_2(qB + B) = h + \nu_2(q + 1)$. Hence
\[
  \Delta_m(qB + B - 1) = \gamma_m(h + \nu_2(q + 1)).
\]
The set $\{0 \le s \le h + \nu_2(q + 1) : s \,\&\, M = 0\}$ splits into the part with $s < h$ and the part with $s \ge h$. The first part has cardinality $\#R_m$ modulo $2$, equal to $\eta_m$. For the second part, since $M < h = 2^k$, the high bit $h$ does not interact with the bits of $M$, so $s \,\&\, M = 0$ if and only if $(s - h) \,\&\, M = 0$. The second part is therefore in bijection with $\{0 \le s' \le \nu_2(q + 1) : s' \,\&\, M = 0\}$, of cardinality $\gamma_m(\nu_2(q + 1)) = \Delta_m(q) \pmod 2$. Adding the two parities,
\[
  \Delta_m(qB + B - 1) = \eta_m \oplus \Delta_m(q). \qedhere
\]
\end{proof}

The derivative formula immediately implies the absence of the factor $00$ in $\Delta_m$.

\begin{lemma}
\label{lem:no-00}
For every $m \ge 1$, the sequence $\Delta_m$ does not contain the factor $00$.
\end{lemma}

\begin{proof}
For every $n \ge 0$, we have
\[
  \nu_2(2n + 1) = 0.
\]
By Lemma~\ref{lem:derivative-formula},
\[
  \Delta_m(2n) = \gamma_m(0) = 1.
\]
Among two consecutive positions, one is even. It therefore carries the letter $1$. The factor $00$ cannot appear in $\Delta_m$.
\end{proof}

The complexity $p_m$ reduces to that of $\Delta_m$ by the following relation.

\begin{lemma}
\label{lem:complexity-derivative-general}
Let $q_m(L) = p_{\Delta_m}(L)$ be the factor complexity of $\Delta_m$, with the convention $q_m(0)=1$. Then, for every $n \ge 1$ and every $m \ge 1$,
\[
        p_m(n) = 2q_m(n-1).
\]
\end{lemma}

\begin{proof}
To every factor
\[
  u = u_0 u_1 \cdots u_{n-1},
\]
of $a_m$, we associate its derived word
\[
  D(u) = (u_0 \oplus u_1) \cdots (u_{n-2} \oplus u_{n-1}),
\]
which is a factor of length $n-1$ of $\Delta_m$. Conversely, every factor of length $n-1$ of $\Delta_m$ appears as the derivative of the factor of length $n$ of $a_m$ located at the same positions.

A binary word of length $n-1$ has exactly two binary preimages of length $n$ under the map $D$. These two preimages are complementary to each other. We verify that the set of factors of $a_m$ is stable under complementation.

Let
\[
  u = a_m(t) \cdots a_m(t+n-1),
\]
be a factor of $a_m$. Since $M$ is fixed, choose an integer $e$ so large that $J = 2^e$ satisfies $J > M$ and $2^J > t + n$. The binary representation of $J$ consists of a single $1$-bit at position $e$, which is strictly higher than the position of any $1$-bit of $M$. Hence $J \,\&\, M = 0$. For $0 \le j < n$, the addition of $2^J$ to $t + j$ produces no carry in positions below $J$, and adds exactly one digit $1$ at position $J$. This bit contributes to the formula, since $J \,\&\, M = 0$. Thus
\[
  a_m(t + j + 2^J) = a_m(t + j) \oplus 1.
\]
The complementary factor $\overline{u}$ therefore appears at the positions shifted by $2^J$. The two preimages of each factor of $\Delta_m$ are thus factors of $a_m$, and the relation
\[
  p_m(n) = 2 q_m(n-1),
\]
follows.
\end{proof}

The recurrence~\eqref{eq:univ-r0}-\eqref{eq:univ-rne0} follows if the sequence $\Delta_m$ satisfies the following synchronization property.

\begin{definition}
Let $m \ge 1$, let $\Delta_m$ be the derivative of $a_m$, and set $k = \lceil \log_2 m \rceil$, $B = 2^{2^k}$. We say that $\Delta_m$ \emph{synchronizes at delay} $2B - 1$ if every factor of $\Delta_m$ of length at least $2B - 1$ determines uniquely its starting position modulo $B$.
\end{definition}

The word ``synchronization'' here means that long enough factors determine their starting phase modulo $B$. It is not the notion of $k$-synchronized function in the sense of Carpi.

Under the synchronization hypothesis, the desubstitution operation is well defined for sufficiently long factors.

\begin{lemma}
\label{lem:desubstitution-generale}
Suppose that $\Delta_m$ synchronizes at delay $2B-1$. For every $L \ge 2B-1$, writing $L = aB+r$ with $0 \le r < B$, we have
\[
q_m(L) = (B-r)q_m(a) + r q_m(a+1).
\]
\end{lemma}

\begin{proof}
For $0 \le s < B$, denote by $\mathcal{F}_s(L)$ the set of factors of $\Delta_m$ of length $L$ starting at a position of phase $s$ modulo $B$. The synchronization hypothesis at delay $2B-1$ implies that the sets $\mathcal{F}_s(L)$ are pairwise disjoint for $L \ge 2B-1$. Indeed, if the same factor belonged to two sets $\mathcal{F}_s(L)$ and $\mathcal{F}_{s'}(L)$, with $s \ne s'$, it would have two possible starting phases modulo $B$.

We recall the form of the blocks of $\Delta_m$. By Lemma~\ref{lem:block-form-derivative}, the first $B-1$ letters of a block of length $B$ are imposed by the fixed pattern $\theta_m$. The last letter of the block of index $q$ has the form
\[
  \eta_m \oplus \Delta_m(q),
\]
where $\eta_m = \#R_m \pmod 2$. Thus, the letters interior to the blocks are fixed, while the end-of-block letters are obtained from consecutive letters of $\Delta_m$, with possibly a fixed complementation.

A factor of $\mathcal{F}_s(L)$ encounters exactly
\[
  t_s = \left\lfloor \frac{L+s}{B} \right\rfloor
\]
end-of-block positions. For $L \ge 2B - 1$, we have $t_s \ge 1$ for every phase $s$, so no degenerate case arises. These positions carry, for some index $q$, the letters
\[
  \eta_m \oplus \Delta_m(q), \;
  \eta_m \oplus \Delta_m(q+1), \;
  \dots, \;
  \eta_m \oplus \Delta_m(q + t_s - 1).
\]
All other letters are fixed by the internal pattern $\theta_m$ and by the phase $s$.

We make the bijection explicit. Define $\Phi_s$ on $\mathcal{F}_s(L)$ by deleting from each factor the internal letters and keeping only the letters at end-of-block positions, after applying the fixed XOR by $\eta_m$. Then $\Phi_s$ maps $\mathcal{F}_s(L)$ onto the set of factors of $\Delta_m$ of length $t_s$. The inverse takes a factor of length $t_s$ in $\Delta_m$, places its successive letters (after XOR by $\eta_m$) at the end-of-block positions corresponding to phase $s$, and fills the remaining positions with the pattern $\theta_m$. We obtain
\[
  |\mathcal{F}_s(L)| = q_m(t_s).
\]

If $L = aB + r$, with $0 \le r < B$, then
\[
  t_s = \left\lfloor \frac{aB + r + s}{B} \right\rfloor.
\]
We thus have $t_s = a$ when $s + r < B$, and $t_s = a+1$ when $s + r \ge B$. There are $B - r$ phases in the first case and $r$ phases in the second. Since synchronization makes the phase classes disjoint, we obtain
\[
  q_m(L) = \sum_{s=0}^{B-1} |\mathcal{F}_s(L)| = (B - r) q_m(a) + r \, q_m(a+1).
\]
\end{proof}

Under the synchronization hypothesis, the system~\eqref{eq:univ-r0}-\eqref{eq:univ-rne0} then follows from this desubstitution.

\begin{proposition}
\label{prop:dnc-conditional}
Let $m \ge 1$ and $B = 2^{2^{\lceil \log_2 m \rceil}}$. If $\Delta_m$ synchronizes at delay $2B-1$, then the complexity $p_m$ satisfies the system \eqref{eq:univ-r0}-\eqref{eq:univ-rne0} for every $n \ge 2$.
\end{proposition}

\begin{proof}
We use the relation
\[
  p_m(n) = 2 q_m(n-1),
\]
from Lemma~\ref{lem:complexity-derivative-general}.

For $r = 0$, we evaluate
\[
  p_m(Bn) = 2 q_m(Bn - 1).
\]
Since $n \ge 2$, we have $Bn - 1 \ge 2B - 1$. We can therefore apply Lemma~\ref{lem:desubstitution-generale}. The Euclidean decomposition is
\[
  Bn - 1 = (n-1) B + (B - 1).
\]
Thus
\[
  q_m(Bn - 1) = q_m(n - 1) + (B - 1) q_m(n).
\]
Multiplying by $2$, then using again $p_m(t) = 2 q_m(t-1)$, we obtain
\[
  p_m(Bn) = p_m(n) + (B - 1) p_m(n + 1),
\]
which gives \eqref{eq:univ-r0}.

Now let $1 \le r \le B - 1$. We evaluate
\[
  p_m(Bn + r) = 2 q_m(Bn + r - 1).
\]
Since $n \ge 2$, we have $Bn + r - 1 \ge 2B$, so the condition $L \ge 2B - 1$ is satisfied. The Euclidean decomposition is
\[
  Bn + r - 1 = nB + (r - 1).
\]
Lemma~\ref{lem:desubstitution-generale} gives
\[
  q_m(Bn + r - 1) = (B - (r - 1)) q_m(n) + (r - 1) q_m(n + 1).
\]
In other words,
\[
  q_m(Bn + r - 1) = (B + 1 - r) q_m(n) + (r - 1) q_m(n + 1).
\]
Multiplying by $2$, we obtain
\[
  p_m(Bn + r) = (B + 1 - r) p_m(n + 1) + (r - 1) p_m(n + 2),
\]
which gives \eqref{eq:univ-rne0}.
\end{proof}

\subsubsection{Synchronization for \texorpdfstring{$m = 2^k$}{m = powers of two}}

Consider the case $m = 2^k$ for $k \ge 0$. Here $M = 2^k-1$, $h = 2^k$ and $B = 2^h$. The condition imposes $s\,\&\,(2^k-1) = 0$, so only positions $s$ that are multiples of $h$ are read. The derivative $\Delta_m$ has a simpler form, given by the following lemma.

\begin{lemma}
\label{lem:derived-morphism-primary}
Let $m = 2^k$ with $k \ge 0$, and let $B = 2^{2^k}$. The sequence $\Delta_m$ is the fixed point starting with $1$ of the $B$-uniform morphism $\Theta_m$ defined on $\{0,1\}$ by
\[
  \Theta_m(c) = \theta_m \cdot (\eta_m \oplus c),
\]
with $\theta_m = 1^{B-1}$ and $\eta_m = 1$. Equivalently,
\[
  \Theta_m(1) = 1^{B-1} 0, \qquad \Theta_m(0) = 1^B.
\]
\end{lemma}

\begin{proof}
We use Lemma~\ref{lem:derivative-formula}. Here
\[
  M = 2^k - 1, \qquad h = 2^k, \qquad B = 2^h.
\]
The condition $s \,\&\, M = 0$ is equivalent to saying that $s$ is a multiple of $h$. Thus
\[
  \gamma_m(t) = \#\{0 \le s \le t : h \mid s\} \pmod 2.
\]

For $0 \le j \le B - 2$, we have $j + 1 < B = 2^h$, so
\[
  \nu_2(j + 1) < h.
\]
The only multiple of $h$ in the interval $[0, \nu_2(j+1)]$ is then $0$. Consequently
\[
  \Delta_m(qB + j) = 1 \qquad (0 \le j \le B - 2).
\]
The internal pattern is therefore $\theta_m = 1^{B-1}$.

For the last letter of the block, take $j = B - 1$, so that we have
\[
  \nu_2(qB + B) = \nu_2(B(q + 1)) = h + \nu_2(q + 1).
\]
If $t = \nu_2(q + 1)$, the multiples of $h$ in $[0, h + t]$ are $0$, then the integers $h + s$, where $s$ ranges over the multiples of $h$ in $[0, t]$. Thus
\[
  \gamma_m(h + t) = 1 \oplus \gamma_m(t),
\]
which corresponds to $\eta_m = \gamma_m(h-1) = 1$. Hence
\[
  \Delta_m(qB + B - 1) = 1 \oplus \Delta_m(q),
\]
and the formula stated for $\Theta_m$ follows.
\end{proof}

The position of the zeros of $\Delta_m$ determines the synchronization.

\begin{lemma}
\label{lem:zero-geometry}
Let $m = 2^k$ with $k \ge 0$, and let $B = 2^{2^k}$. In $\Delta_m$, the zeros are isolated, occur only at positions congruent to $B - 1$ modulo $B$, and two consecutive zeros are separated by a distance of $B$ or $2B$. Consequently, $\Delta_m$ synchronizes at delay $2B - 1$.
\end{lemma}

\begin{proof}
By Lemma~\ref{lem:derived-morphism-primary}, the images $\Theta_m(0)$ and $\Theta_m(1)$ start with $1$, and the only possible zero in an image appears at the last position of
\[
  \Theta_m(1) = 1^{B-1} 0.
\]
The zeros of $\Delta_m$ are therefore all located at positions congruent to $B - 1$ modulo $B$.

By Lemma~\ref{lem:no-00}, the sequence $\Delta_m$ contains no factor $00$. Each block contains at most one zero, located at the last position, and the next block always starts with $1$, so two zeros cannot be consecutive.

The zeros of $\Delta_m$ correspond exactly to the letters $1$ of the ancestor $\Delta_m$. Two consecutive letters $1$ of the ancestor are separated by a distance $1$ or $2$, since the ancestor itself contains no factor $00$ (Lemma~\ref{lem:no-00}). After inflation by $\Theta_m$, two consecutive zeros of $\Delta_m$ are therefore separated by a distance $B$ or $2B$.

Synchronization follows from the positions of the zeros. If a factor of length $L \ge 2B - 1$ contains a zero, the position of this zero in the factor determines the starting phase modulo $B$, since all zeros are at phase $B - 1$.

If it contains no zero, then $L$ cannot exceed $2B - 1$, since two consecutive zeros are at distance at most $2B$. When $L = 2B - 1$, the factor is the full block of $1$'s located between two zeros at distance $2B$. It therefore starts just after a zero, that is, at phase $0$ modulo $B$. The phase is therefore again determined. Synchronization at delay $2B - 1$ follows.
\end{proof}

Lemma~\ref{lem:zero-geometry} gives synchronization in the primary case.

\begin{theorem}
\label{thm:dnc-primary}
For every $k \ge 0$, setting $m = 2^k$ and $B = 2^{2^k}$, the derivative
$\Delta_m$ synchronizes at delay $2B-1$.
\end{theorem}

\begin{proof}
This is exactly Lemma~\ref{lem:zero-geometry}.
\end{proof}

\subsubsection{Synchronization in the general case}

For $m$ not a power of two, the last letter of a block equals the ancestral letter without complementation. We establish synchronization using the periodicity of the internal pattern and the absence of the factor $00$.

\begin{lemma}
\label{lem:theta-periodicity}
For every $m \ge 1$, set $M = m-1$, $k = \lceil \log_2 m \rceil$ and $h = 2^k$. Let $\ell_m = \max\{0 \le s < h : s \,\&\, M = 0\}$. The internal pattern $\theta_m$ of length $B-1$, defined by $\theta_m(r) = \gamma_m(\nu_2(r+1))$ for $0 \le r < B-1$, is periodic with period $P_m = 2^{\ell_m}$, and $P_m$ strictly divides $B = 2^h$.

Moreover, $P_m$ is the minimal period of $\theta_m$ as soon as $\ell_m > 0$.
\end{lemma}

\begin{proof}
Set $M = m - 1$. For $0 \le r \le B - 2$, we have
\[
  r + 1 < B = 2^h,
\]
so
\[
  \nu_2(r + 1) < h.
\]
Thus, in the internal pattern, the value
\[
  \theta_m(r) = \gamma_m(\nu_2(r + 1)),
\]
involves only the integers $s \in [0, h - 1]$ such that $s \,\&\, M = 0$.

By definition of $\ell_m$, no integer
\[
  s \in (\ell_m, h - 1],
\]
satisfies $s \,\&\, M = 0$. The function $\gamma_m$ is therefore constant on the interval $[\ell_m, h - 1]$.

Set
\[
  P_m = 2^{\ell_m}.
\]
We prove that $P_m$ is the minimal period of $\theta_m$.
If $0 \le r \le B - 2 - P_m$, set $u = r + 1$. When $\nu_2(u) < \ell_m$, adding $P_m$ does not change the $2$-adic valuation, so
\[
  \nu_2(u + P_m) = \nu_2(u).
\]
When $\nu_2(u) \ge \ell_m$, the two valuations $\nu_2(u)$ and $\nu_2(u + P_m)$ belong to $[\ell_m, h - 1]$, and $\gamma_m$ is constant there. In both cases,
\[
  \theta_m(r + P_m) = \theta_m(r).
\]
The pattern $\theta_m$ is therefore $P_m$-periodic.

Finally, since $0 \le \ell_m < h$, the number $P_m = 2^{\ell_m}$ divides $B = 2^h$ and satisfies $P_m < B$. This includes the case $m = 1$, for which $\ell_1 = 0$ and $P_1 = 1$.

Assume $\ell_m > 0$. Since $\ell_m \in R_m$, we have
\[
  \gamma_m(\ell_m) = \gamma_m(\ell_m - 1) \oplus 1.
\]
Now $\nu_2(2^{\ell_m}) = \ell_m$ and $\nu_2(2^{\ell_m - 1}) = \ell_m - 1$, whence
\[
  \theta_m(P_m - 1) \ne \theta_m(P_m/2 - 1).
\]
Thus no period of $\theta_m$ can divide $P_m / 2$. To exclude a proper period $d < P_m$ that does not divide $P_m$, we invoke the Fine-Wilf theorem: if $d$ and $P_m$ are both periods of the finite word $\theta_m$ of length $B - 1$, and if $B - 1 \ge P_m + d - \gcd(P_m, d)$, then $\gcd(P_m, d)$ is also a period. Now $B - 1 \ge 2 P_m - 1 \ge P_m + d - \gcd(P_m, d)$ since $P_m < B$ and $d < P_m$. Hence $\gcd(P_m, d)$ would be a period dividing $P_m / 2$, contradicting the previous step. The minimal period is therefore indeed $P_m$.
\end{proof}

The morphic structure of the derivative when $m$ is not a power of two follows from the evaluation at the transition indices between blocks.

\begin{lemma}
\label{lem:transition-morphic}
Let $m \ge 1$ not a power of two, with $k = \lceil \log_2 m \rceil$, $h = 2^k$ and $B = 2^h$. The derivative sequence $\Delta_m$ is the fixed point starting with $1$ of the $B$-uniform morphism $\Theta_m$ defined on $\{0,1\}$ by
\[
  \Theta_m(c) = \theta_m \cdot (\eta_m \oplus c),
\]
with $\eta_m = \gamma_m(h-1) = 0$. The morphism therefore reduces to $\Theta_m(c) = \theta_m \cdot c$, and $\Theta_m(0)$ as well as $\Theta_m(1)$ start with the internal pattern $\theta_m$.
\end{lemma}

\begin{proof}
We use Lemma~\ref{lem:derivative-formula}. For $0 \le r \le B - 2$, the letter at internal position of the block is
\[
  \Delta_m(qB + r) = \gamma_m(\nu_2(r + 1)) = \theta_m(r).
\]

For the last letter of the block, set
\[
  t = \nu_2(q + 1).
\]
Then
\[
  \nu_2(qB + B) = \nu_2(B(q + 1)) = h + t.
\]
Since $M < h$, the bit at position $k$ of $M$ is zero, while $h = 2^k$ has only one nonzero bit at position $k$. Thus, for $0 \le s \le t$,
\[
  (h + s) \,\&\, M = 0 \quad\Longleftrightarrow\quad s \,\&\, M = 0.
\]
We thus obtain
\[
  \gamma_m(h + t) = \gamma_m(h - 1) \oplus \gamma_m(t).
\]

Suppose now that $m$ is not a power of two. Then $M \ne h - 1$. Among the $k$ low-order bits, at least one bit of $M$ is zero. The number of integers $s \in [0, h - 1]$ such that $s \,\&\, M = 0$ is therefore a positive power of $2$, hence an even integer. Thus
\[
  \eta_m = \gamma_m(h - 1) = 0.
\]
It follows that
\[
  \Delta_m(qB + B - 1) = \gamma_m(t) = \Delta_m(q),
\]
which gives $\Delta_m = \Theta_m(\Delta_m)$ with $\Theta_m(c) = \theta_m c$.
\end{proof}

Synchronization relies on the following property of translations on $\mathbb{Z} / B \mathbb{Z}$.

\begin{lemma}
\label{lem:translation-perforee}
Let $m \ge 1$ with $m$ not a power of two, and take again the notations $h = 2^k$, $B = 2^h$, $P_m = 2^{\ell_m}$. For every $\delta \in \{1, \dots, B - 1\}$, if a function $f \colon \{0, \dots, B - 2\} \to \{0, 1\}$ satisfies
\begin{itemize}
  \item[\rm(i)] for every $x \in \{0, \dots, B - 2\}$ such that $x + \delta \not\equiv B - 1 \pmod B$, we have $f(x) = f((x + \delta) \bmod B)$,
  \item[\rm(ii)] $f$ is $P_m$-periodic,
\end{itemize}
then
\[
  f(\delta - 1) = f(P_m - 1).
\]
\end{lemma}

\begin{proof}
The translation $x \mapsto x + \delta \pmod B$ on $\mathbb{Z} / B \mathbb{Z}$ decomposes into $g' = \gcd(\delta, B)$ cycles, each of length $B / g'$.

If $B / g' \ge 3$, removing the vertex $B - 1$ transforms the cycle containing it into a path connecting the other vertices of that cycle. Hypothesis {\rm(i)} applies to every edge of this path, so the constancy of $f$ propagates from vertex to vertex along the path. The other cycles remain unchanged, and Hypothesis {\rm(i)} likewise forces $f$ to be constant on each of them. Therefore $f$ is constant on each class modulo $g'$, restricted to $\{0, \dots, B - 2\}$. Combined with {\rm(ii)}, since $g'$ and $P_m$ are powers of $2$, setting $g = \gcd(g', P_m)$, we obtain that $f$ is $g$-periodic on $\{0, \dots, B - 2\}$. Now $g \mid \delta$ and $g \mid P_m$. The indices $\delta - 1$ and $P_m - 1$ are therefore congruent modulo $g$ and belong to $\{0, \dots, B - 2\}$, whence
\[
  f(\delta - 1) = f(P_m - 1).
\]

If $B / g' = 2$, then $\delta = B / 2$. Since $P_m$ strictly divides $B$, $P_m \le B / 2$ and so $P_m \mid \delta$. The $P_m$-periodicity {\rm(ii)} directly gives $f(\delta - 1) = f(P_m - 1)$.

The case $B / g' = 1$ would give $\delta \equiv 0 \pmod B$, which is excluded by $\delta \in \{1, \dots, B - 1\}$.
\end{proof}

These four lemmas give the synchronization in the general case.

\begin{theorem}
\label{thm:dnc-general}
For every $m \ge 1$, the derivative $\Delta_m$ synchronizes at delay
$2B - 1$, with $B = 2^{2^{\lceil \log_2 m \rceil}}$.
\end{theorem}

\begin{proof}
The synchronization at delay $2B - 1$ for $m = 2^k$ is established in Lemma~\ref{lem:zero-geometry}.

Suppose now that $m$ is not a power of two. We use the morphism
\[
  \Theta_m(c) = \theta_m c,
\]
from Lemma~\ref{lem:transition-morphic}. We show synchronization at delay $2B - 1$. It is enough to exclude two occurrences of the same factor of length exactly $2B - 1$ with distinct starting phases modulo $B$. Indeed, any longer common factor with distinct starting phases has a prefix of length $2B - 1$ with the same property.

Suppose by contradiction that the same factor $W$ of length $2B - 1$ appears with two distinct phases modulo $B$. We choose one of the two occurrences with a nonzero phase, which is possible since the phases are distinct. Denote this phase $b$, with
\[
  1 \le b \le B - 1.
\]
The other phase is then $b + \delta$ modulo $B$, with
\[
  1 \le \delta \le B - 1.
\]

In the first occurrence, the two transition positions contained in the window are
\[
  j_0 = B - 1 - b \qquad\text{and}\qquad j_0 + B.
\]
They correspond to the last letters of two consecutive blocks, so they carry
\[
  \Delta_m(q) \qquad\text{and}\qquad \Delta_m(q + 1),
\]
for some $q$.

In the second occurrence, these same two positions have internal offset
\[
  \delta - 1,
\]
modulo $B$. Since $0 \le \delta - 1 \le B - 2$, this offset is internal to the block. The value there is therefore imposed by $\theta_m$. We obtain
\[
  \Delta_m(q) = \theta_m(\delta - 1), \qquad \Delta_m(q + 1) = \theta_m(\delta - 1).
\]
We show that
\[
  \theta_m(\delta - 1) = 0.
\]

We compare the two occurrences at internal offsets. For every $x \in \{0, \dots, B - 2\}$, choose $j_x \in \{0, \dots, 2B - 2\}$ such that $b + j_x \equiv x \pmod B$. Such a $j_x$ exists because an interval of length $2B - 1$ meets every residue class modulo $B$ at least once. In particular, the residue $b - 1$ is met at $j_x = B - 1$. At this position the first occurrence is internal. If moreover $x + \delta \not\equiv B - 1 \pmod B$, then the second occurrence is internal as well at the same position $j_x$, and the equality of the two factors at $j_x$ gives
\[
  \theta_m(x) = \theta_m((x + \delta) \bmod B).
\]
The function $\theta_m$ is $P_m$-periodic (Lemma~\ref{lem:theta-periodicity}). Lemma~\ref{lem:translation-perforee} applied to $f = \theta_m$ therefore gives
\[
  \theta_m(\delta - 1) = \theta_m(P_m - 1).
\]
Since $P_m = 2^{\ell_m}$, we have
\[
  \theta_m(P_m - 1) = \gamma_m(\nu_2(P_m)) = \gamma_m(\ell_m).
\]
For $m$ not a power of two, we saw in the proof of Lemma~\ref{lem:transition-morphic} that
\[
  \gamma_m(h - 1) = 0.
\]
Since $\gamma_m$ is constant on $[\ell_m, h - 1]$, it follows that
\[
  \gamma_m(\ell_m) = 0.
\]
Hence
\[
  \theta_m(\delta - 1) = 0.
\]

The transition letters between the two phase-distinct occurrences are precisely the letters of $\Delta_m$ at consecutive positions $q$ and $q+1$. The relation $\theta_m(\delta - 1) = 0$ forces both transition letters to be $0$, so
\[
  \Delta_m(q) = \Delta_m(q + 1) = 0.
\]
The factor $00$ thus appears in $\Delta_m$, contradicting Lemma~\ref{lem:no-00}. This proves the synchronization at delay $2B-1$.
\end{proof}

\begin{proof}[Proof of Theorem~\ref{thm:complexity-recurrence}]
By Theorem~\ref{thm:dnc-general}, the derivative $\Delta_m$ synchronizes
at delay $2B-1$. Proposition~\ref{prop:dnc-conditional} then gives the
recurrences \eqref{eq:univ-r0} and \eqref{eq:univ-rne0} for every
$n\ge2$.
\end{proof}

\subsection{Initial values}
\label{subsec:initial-values}

The recurrence of Theorem~\ref{thm:complexity-recurrence} reduces the
computation of \(p_m\) to the finite seed
\[
  p_m(1),p_m(2),\ldots,p_m(2B-1).
\]
For a fixed \(m\), this finite list can be computed directly from the
binary digit formula of Theorem~\ref{thm:masque-binaire}, or from the
uniform morphism of Section~\ref{sec:morphic}. In the primary case
\(m=2^k\), the seed has the following closed form.

\begin{proposition}
\label{prop:delta-initial-complexity}
Let $m = 2^k$ with $k \ge 0$, and let $B = 2^{2^k}$. The complexity $q_m(L)$ satisfies $q_m(L) = L + 1$ for $1 \le L \le B$, and $q_m(L) = 2L - B + 1$ for $B < L \le 2B - 1$.
\end{proposition}

\begin{proof}
By Lemma~\ref{lem:zero-geometry}, the zeros are isolated and two consecutive zeros are separated by a distance $B$ or $2B$. Moreover, both distances occur. The factor $11$ appears in the fixed point of $\Theta_m$. If $B > 2$, it appears inside $\Theta_m(1) = 1^{B-1} 0$. If $B = 2$, then $\Theta_m(0) = 11$ produces the factor $11$, and the fixed point contains the letter $0$ since $\Theta_m(1)$ ends in $0$. The inflation of a pair $11$ in the ancestor gives two zeros separated by $B$. The factor $101$ appears around any zero of the fixed point, since zeros are isolated and lie at the end of $\Theta_m(1)$. Its inflation gives two zeros separated by $2B$.

Let first $1 \le L \le B$. A factor of length $L$ contains at most one zero. The only possible words are therefore
\[
  1^L,
\]
and
\[
  1^i 0 1^{L-1-i} \qquad (0 \le i \le L - 1).
\]
They all appear by sliding a window of length $L$ around a zero. Since the neighboring zeros are always at distance at least $B$, there are at least $B - 1$ letters $1$ on each side, which suffices since $L \le B$. We thus obtain
\[
  q_m(L) = L + 1.
\]

Suppose now
\[
  B < L \le 2B - 1.
\]
A factor of length $L$ contains at most two zeros.

There is a single factor with no zero, namely $1^L$. It exists because a distance $2B$ between two zeros provides a block of $2B - 1$ letters $1$.

The factors with a single zero are the words
\[
  1^i 0 1^{L-1-i} \qquad (0 \le i \le L - 1).
\]
They all appear. If $i \le B - 1$ and $L - 1 - i \le B - 1$, any zero suffices. If $L - 1 - i > B - 1$, we take a zero followed by a gap $2B$. If $i > B - 1$, we take a zero preceded by a gap $2B$. These three cases cover all indices $i$, since $L \le 2B - 1$.

The factors with two zeros must have their two zeros separated by $B$, since the length is at most $2B - 1$. They are therefore of the form
\[
  1^i 0 1^{B-1} 0 1^{L-B-1-i} \qquad (0 \le i \le L - B - 1).
\]
These words all appear from a pair of zeros separated by $B$. The outer sides require at most $B - 2$ letters $1$, and the neighboring zeros are at distance at least $B$.

We therefore have
\[
  1 + L + (L - B) = 2L - B + 1,
\]
distinct factors, which gives
\[
  q_m(L) = 2L - B + 1.
\]
\end{proof}

The closed-form formula for $p_m$ follows.

\begin{corollary}
\label{cor:primary-closed-complexity}
For $m=2^k$, $k\ge 0$ and $B=2^{2^k}$, we have $p_m(n)=2n$ for $1\le n\le B+1$. For every integer $i\ge 0$, the complexity is determined by
\[
  p_m(n)=
  \begin{cases}
    -2(B-1)B^i+4(n-1),
       & \text{if } B^{i+1}+1<n\le (2B-1)B^i+1,\\[2mm]
    2B^{i+1}+2(n-1),
       & \text{if } (2B-1)B^i+1<n\le B^{i+2}+1.
  \end{cases}
\]
\end{corollary}

\begin{proof}
Set
\[
  L = n - 1.
\]
By Lemma~\ref{lem:complexity-derivative-general}, we have
\[
  p_m(n) = 2 q_m(L).
\]
The formula
\[
  p_m(n) = 2n \qquad (1 \le n \le B + 1),
\]
comes from $q_m(L) = L + 1$ for $0 \le L \le B$, with the convention $q_m(0) = 1$.

For $q_m(L)$, the formulas are as follows. Proposition~\ref{prop:delta-initial-complexity} gives
\[
  q_m(L) = 2L - (B - 1) \qquad (B < L \le 2B - 1).
\]
This gives the first formula for $i = 0$.

For $2B - 1 < L \le B^2$, write $L = aB + r$, with $0 \le r < B$. Then $2 \le a \le B$, and if $a = B$ we necessarily have $r = 0$. Desubstitution gives
\[
  q_m(L) = (B - r) q_m(a) + r q_m(a + 1).
\]
Since $a \le B$, the initial values give
\[
  q_m(a) = a + 1, \qquad q_m(a + 1) = a + 2,
\]
when the second term occurs. Thus
\[
  q_m(L) = B(a + 1) + r = L + B.
\]
This gives the second formula for $i = 0$.

Suppose now the two formulas known up to rank $i - 1$. We prove those at rank $i$. The known formulas extend to the boundary values in the form
\[
  q_m(B^i) = B^i + B^{i-1},
\]
\[
  q_m(x) = -(B - 1) B^{i-1} + 2x \qquad (B^i \le x \le (2B - 1) B^{i-1}),
\]
and
\[
  q_m(x) = B^i + x \qquad ((2B - 1) B^{i-1} \le x \le B^{i+1}).
\]

Let first
\[
  B^{i+1} < L \le (2B - 1) B^i.
\]
Write $L = aB + r$, with $0 \le r < B$. The integers $a$ and $a + 1$, when the second occurs, belong to the first formula extended from rank $i - 1$. Thus
\[
  q_m(a) = -(B - 1) B^{i-1} + 2a,
\]
and
\[
  q_m(a + 1) = -(B - 1) B^{i-1} + 2a + 2.
\]
Desubstitution then gives
\[
\begin{aligned}
  q_m(L) &= (B - r) \bigl(-(B - 1) B^{i-1} + 2a\bigr) \\
         &\quad + r \bigl(-(B - 1) B^{i-1} + 2a + 2\bigr) \\
         &= -(B - 1) B^i + 2(aB + r) \\
         &= -(B - 1) B^i + 2L.
\end{aligned}
\]

Let next
\[
  (2B - 1) B^i < L \le B^{i+2}.
\]
Write again $L = aB + r$, with $0 \le r < B$. The integers $a$ and $a + 1$, when the second occurs, belong to the second formula extended from rank $i - 1$. Thus
\[
  q_m(a) = B^i + a,
\]
and
\[
  q_m(a + 1) = B^i + a + 1.
\]
Desubstitution gives
\[
\begin{aligned}
  q_m(L) &= (B - r)(B^i + a) + r(B^i + a + 1) \\
         &= B^{i+1} + aB + r \\
         &= B^{i+1} + L.
\end{aligned}
\]

We thus have, for every $i \ge 0$,
\[
  q_m(L) = -(B - 1) B^i + 2L \qquad (B^{i+1} < L \le (2B - 1) B^i),
\]
and
\[
  q_m(L) = B^{i+1} + L \qquad ((2B - 1) B^i < L \le B^{i+2}).
\]
Replacing $L$ by $n - 1$ and multiplying by $2$, we obtain the stated formula for $p_m(n)$.
\end{proof}

These formulas immediately specialize to the Layman sequence.

\begin{example}
\label{ex:complexity-a2}
For $m = 2$, we have $k = 1$ and $B = 4$. Therefore $p_2(n) = 2n$ for $1 \le n \le 5$, and for every $i \ge 0$, Corollary~\ref{cor:primary-closed-complexity} specializes to
\[
  p_2(n) =
  \begin{cases}
    -6 \cdot 4^i + 4(n-1),
       & \text{if } 4^{i+1}+1 < n \le 7 \cdot 4^i + 1,\\[2mm]
    8 \cdot 4^i + 2(n-1),
       & \text{if } 7 \cdot 4^i + 1 < n \le 16 \cdot 4^i + 1.
  \end{cases}
\]
The first values are
\[
  p_2(1), p_2(2), \ldots = 2, 4, 6, 8, 10, 14, 18, 22, 24, 26, 28, 30, \ldots.
\]
\end{example}

For a non-primary value of \(m\), the seed is still finite and directly
computable. The next example shows how the recurrence is used once this
seed is known.

\begin{example}
\label{ex:m3-complete}
For \(m=3\), we have
\[
  k=2,\qquad B=16.
\]
A direct computation from the binary digit formula gives the finite seed
\[
  p_3(1),p_3(2),\ldots,p_3(31):
\]
\[
\begin{array}{c|rrrrrrrr}
n        &1&2&3&4&5&6&7&8\\
\hline
p_3(n)   &2&4&6&10&12&14&16&18\\[2mm]
n        &9&10&11&12&13&14&15&16\\
\hline
p_3(n)   &20&22&24&26&28&30&32&34\\[2mm]
n        &17&18&19&20&21&22&23&24\\
\hline
p_3(n)   &36&40&44&48&52&56&60&64\\[2mm]
n        &25&26&27&28&29&30&31\\
\hline
p_3(n)   &68&72&76&80&84&88&92
\end{array}
\]
For all larger values, Theorem~\ref{thm:complexity-recurrence} applies.
If
\[
  N=16q+r,\qquad q\ge2,\qquad 0\le r<16,
\]
then
\[
  p_3(16q)=p_3(q)+15p_3(q+1),
\]
and, for \(1\le r\le15\),
\[
  p_3(16q+r)=(17-r)p_3(q+1)+(r-1)p_3(q+2).
\]
For example,
\[
  p_3(32)=p_3(2)+15p_3(3)=4+15\cdot6=94.
\]
Together with the seed above, the recurrence determines all values of
\(p_3\).
\end{example}

\section*{Perspectives}

Two extensions will be treated elsewhere. The first is a $q$-ary version of the transform, where addition modulo $q$ replaces binary complementation. For $q$ prime, the iterates are fixed points of balanced morphisms on a $q$-letter alphabet in the sense of \v{C}ern\'y \cite{Cerny2013, Cerny2017}. This gives $q$-ary Prouhet-Tarry-Escott partitions, composition formulas, and factor complexity formulas. The second extension concerns Woods-Robbins type products attached to the iterates $a_{2^k}$.

The transform $\mathcal{T}$ can also be applied to other binary words. For example, if $f$ is the Fibonacci word \seqnum{A003849}, fixed by $0 \mapsto 01$ and $1 \mapsto 0$, then
\[
  \mathcal{T}(f) = \mathrm{ftm},
\]
where $\mathrm{ftm}$ \seqnum{A095076} is the Fibonacci-Thue-Morse sequence, that is, the parity of the digit sum in the Zeckendorf representation. For this sequence, see Ferrand \cite{Ferrand2007} and Shallit \cite{ShallitFibo2022, ShallitFiboComplexity2021}. It is Fibonacci-automatic in the sense of Mousavi, Schaeffer, and Shallit \cite{MousaviSchaefferShallit2016}. We do not pursue this direction here.

\end{document}